\documentclass[12pt,a4paper,oneside,headings=normal]{scrartcl}
\usepackage[left=3cm,right=3cm,top=2cm,bottom=2.5cm]{geometry}			
\usepackage[table,xcdraw]{xcolor}
\usepackage[framed, numbered]{matlab-prettifier}
\usepackage{setspace}													
\usepackage[bottom,hang]{footmisc}										
\usepackage[T1]{fontenc}												
\usepackage[utf8x]{inputenc}											
\usepackage[ngerman,english]{babel}									
\usepackage{mathrsfs,amsmath}
\allowdisplaybreaks
\usepackage{amsthm}														%LATEX-Mathematik-Formatierungen und -Symbole
\usepackage{amstext}													%LATEX-Mathematik-Formatierungen und -Symbole
\usepackage{amssymb}													%LATEX-Mathematik-Formatierungen und -Symbole
\usepackage{amsfonts}													%LATEX-Mathematik-Formatierungen und -Symbole
\usepackage{dsfont} 													
\usepackage{mathtools}
\usepackage[numbers, square]{natbib}

\usepackage{url}														
\usepackage{graphicx}													
\usepackage[bf,hang,nooneline,justification=centering %justification auskommentiert, damit die Bildunterschriften nicht zentriert sind 
]{caption}			
\usepackage{tabularx}													
\usepackage{multirow}													
\usepackage{booktabs}													
\usepackage{acronym}
\usepackage[titletoc,title]{appendix}								
\usepackage{ucs} 														
\usepackage{placeins}
\usepackage{lmodern} 													
\usepackage{xcolor}														
\usepackage{hyperref}													
\usepackage{stmaryrd}
\usepackage{ marvosym } 
\usepackage{comment}
\def\XXint#1#2#3{{\setbox0=\hbox{$#1{#2#3}{\int}$ }
\vcenter{\hbox{$#2#3$ }}\kern-.6\wd0}}

\usepackage{paralist}
\numberwithin{equation}{section}

\makeatletter
\def\moverlay{\mathpalette\mov@rlay}
\def\mov@rlay#1#2{\leavevmode\vtop{%
   \baselineskip\z@skip \lineskiplimit-\maxdimen
   \ialign{\hfil$\m@th#1##$\hfil\cr#2\crcr}}}
\newcommand{\charfusion}[3][\mathord]{
    #1{\ifx#1\mathop\vphantom{#2}\fi
        \mathpalette\mov@rlay{#2\cr#3}
      }
    \ifx#1\mathop\expandafter\displaylimits\fi}
\makeatother

\newcommand{\bigcupdot}{\charfusion[\mathop]{\bigcup}{\cdot}}

\theoremstyle{definition} 
\newtheorem{theorem}{Theorem}[section]                      
\newtheorem{lemma}[theorem]{Lemma}
\newtheorem{definition}[theorem]{Definition}

\newtheorem{proposition}[theorem]{Proposition}

\newtheorem{remark}[theorem]{Remark}

\newtheorem{ex}[theorem]{Example}

\newenvironment{beweis}{\begin{proof}[Proof]}{\end{proof}}

\DeclareMathOperator*{\Exp}{Exp}

\DeclareMathOperator{\ccc}{C}

\newcommand{\N}{\mathbb{N}}
\newcommand{\Ex}{\mathbb{E}}
\newcommand{\PR}{\mathbb{P}}

\newcommand{\R}{\mathbb{R}}

\renewcommand{\phi}{\varphi}
\renewcommand{\epsilon}{\varepsilon}
\newcommand{\eps}{\varepsilon}

\newcommand*{\f}{\mathcal{F}}

\DeclareMathOperator{\spt}{spt}

\newcommand{\ttt}[1]{\text{#1}}

\newcommand{\ind}[1]{\mathds{ 1 }_{\{{#1}\}}}
\newcommand{\inda}[1]{\mathds{ 1 }_{{#1}}}

\newcommand{\edp}{e_{D,\Psi}}
\graphicspath{ {Bilder/} }
\usepackage{tikz}
\usetikzlibrary{arrows,automata,shapes,calc}
\usepackage{setspace}

\title{On the time consistent solution to optimal stopping problems with expectation constraint}
\author{S. Christensen\thanks{Kiel University, Mathematical Department, \emph{Email:} \href{mailto:christensen@math.uni-kiel.de}{christensen@math.uni-kiel.de} } , 
M. Klein\thanks{Kiel University, Mathematical Department, \emph{Email:} \href{mailto:maike.klein@math.uni-kiel.de}{maike.klein@math.uni-kiel.de}} , 
B. Schultz\thanks{Kiel University, Mathematical Department, \emph{Email:} \href{mailto:schultz@math.uni-kiel.de}{schultz@math.uni-kiel.de}\\[0.1cm]
We thank two anonymous referees for carefully reading the manuscript and highly appreciate their comments, which contributed to improve our paper.}
}
\date{\today}

\begin{document}

	\maketitle
	\begin{abstract}
    We study the (weak) equilibrium problem arising from the problem of optimally stopping a one-dimensional diffusion subject to an expectation constraint on the time until stopping. The weak equilibrium problem is realized with a set of randomized but purely state dependent stopping times as admissible strategies. We derive a verification theorem and necessary conditions for equilibria, which together basically characterize all equilibria. Furthermore, additional structural properties of equilibria are obtained to feed a possible guess-and-verify approach, which is then illustrated by an example.
    \end{abstract}

	\begin{center}\footnotesize
		\begin{tabular}{l@{ : }p{8.5cm}}
			{\itshape 2020 MSC} & 60G40, 60J70, 91A25,91B51.
			\\
			{\itshape Keywords} & Optimal stopping, expectation constraint, equilibrium, time consistent solution.
		\end{tabular}
	\end{center}

\section{Introduction}
In this paper we consider the stopping problem of maximizing the reward functional $\Ex_x[ e^{-r\tau}g(X_\tau)]$ subject to $\Ex_x[\tau] \le T$, where $X=(X_t)_{t \in [0,\infty)}$ is a regular Itô diffusion on an interval $I\subset \R$, $x$ is the starting value of $X$, $g$ is a payoff function, $r\ge0 $ is a discount factor and $T>0$ is a time constraint.

This problem generalizes optimal stopping with fixed time horizon to include situations where, e.g.\ due to uncertainty in planning, such a time constraint seems unreasonable although a tractable notion of a time limitation should be imposed. Particularly if such problems occur \textit{repeatedly} and \textit{independently} it appears to be a natural choice to bound the average time spend. Optimal stopping problems with expectation constraints were first studied %in a precommitment sense 
by Kennedy \cite{kennedy1982} in 1982 and later on were analyzed in e.g.\ \cite{miller2017, AKK, 3points, BayraktarYao2020, BayraktarYao2023}.

Most recently, the question of time-(in)consistency of optimal stopping and more general stochastic control problems has increasingly become the focus of mathematical discussion, cf.\ \cite{bjoerk2021time}. The main idea is that in practical optimal stopping an agent observes a single path $t\mapsto X_t(\omega)$ of the state process and makes a stopping decision at each point in time based on her past observations. This gives the agent the ability to reconsider her initial strategy (consisting of a stopping time) for maximizing her utility, if her preferences change in the future. This idea traces back to the concept of \textit{consistent planning} pioneered by Strotz\ \cite{strotz}.

In the following, we elaborate the (rather far-reaching) implications of this comparatively simple observation for the problem with expectation constraint, but first we take a closer look at how the expectation constraint turns the classical, time-consistent problem of maximizing $\Ex_x[ e^{-r\tau}g(X_\tau)]$ (without the constraint and infinite time horizon) into a time-inconsistent one. 

An agent facing this problem by observing the state process and acting upon it would, during her time of observing, i.e.\ for all $t$ up to the time she stops, try to solve the problem of maximizing $e^{-rt}\Ex_{X_t(\omega)}[e^{-r\tau_t} g(X_{\tau_t})]$ over all stopping times $\tau_t \ge 0$. By the Markovian nature of the state process $X$, these problems are all structurally the same. This leads to the well known fact that for the unconstrained problem it is optimal to make the stopping decision solely based on the current state of the process, i.e.\ \textit{the optimal stopping time is a first exit time}.

The situation changes if we introduce the expectation constraint. Even though the optimization target at time $t$ is still given by $e^{-rt}\Ex_{X_t(\omega)}[e^{-r\tau_t} g(X_{\tau_t})]$ we can now only choose stopping times $\tau_t$ such that the expectation constraint is met. Thus, to implement an optimal stopping strategy in practice, an agent who makes her stopping decision at each point in time based on her single observation of the state process and the \textit{stopping strategy} up to that point faces the problem that she must still satisfy her initial constraint of $\Ex_x[\tau]\le T$ even though some time has passed.

It is common to assume that the agent displays exactly this type of precommitment behavior in the mathematical literature (see discussion below), while in practice this seems to be the exception rather than the rule. To highlight that this problem setting is particularly prone to such irrationalities note that the expectation constraint allows for \textit{sacrificial strategies}, such as to stop early in some instances, or in other words \textit{sacrifice} potential future gain, in order to save up time to wait for a maximal gain on other occasions and still meet the expectation constraint. Moreover, the stopping decision must not even be solely based on the state process but could also rely on some external randomization device. Even if such a strategy is deemed optimal at first, in practice it seems very likely that at the time such a strategy calls for stopping an agent will fail to commit to it. In the agent's mind, it now seems preferable to keep the process alive for some potential future benefit and sacrifice another \textit{future run} rather than satisfy the constraint. 

Due to the Markovian nature of the state process, especially when the time horizon $T$ is rather large compared to the elapsed time, it seems even more likely that from the agents' point of view the problems at time 0 and at time $t$ are largely the same, up to the current state of $X$, and should be treated as such. This means at time $t$ the constraint is perceived as 
\begin{align*}
    \Ex_{X_t(\omega)}[\tau_t] \le T  
\end{align*}
and not as $\Ex_x[\ind{\tau > t}(\tau_t +t)] \le T- \Ex_x[\ind{\tau\le t} \tau]$. Embedding that into the space time frame and accounting for all possible states of $X_t$ we end up with the constraint $$ \Ex_y[\tau_t]\le T $$
for all $y\in I$ and all $t\in[0,\infty)$. To address the issues of partially rational behavior, such as procrastination the problem is now treated as a game-theoretic (weak) equilibrium problem. We realize the (weak) equilibrium problem with a set of purely state dependent stopping times, see Section \ref{Ch2}.\ This allows us to reformulate the problem as the (weak) equilibrium problem corresponding to the maximization of  $\Ex_x[ e^{-r\tau}g(X_\tau)]$ subject to $\Ex_y[\tau] \le T$ for all $y \in I$. 

\subsection{Structure of the paper}
The exact problem setting as well as the concept of a (weak) equilibrium is formally introduced in Section \ref{Ch2}. In Section \ref{Ch3} we first derive necessary conditions for equilibria. Then we show a verification theorem, see Theorem \ref{veri}, which completes the necessary conditions to basically characterize all equilibria. In Section \ref{Ch4} we propose a scheme to construct equilibria with an example provided in Section \ref{Ch5}.

\subsection{Literature}
Different formulations of equilibrium problems to various problems in optimal stopping/control have been treated by now. In \cite{bjoerk2021time} the authors discuss some general types of reward functionals and payoff functions that induce time inconsistent stopping/control problems. Some of the special cases that received the most attention are so called non-exponential discounting, treated in \cite{ekeland2006noncommitment, BodnariuChristensenLindensjoe2022, huang2017timeconsistent, huang2020, huang2020_2, huang_zhou_diskr,bayraktar2022equilibria, ebert2020_weighteddiscount, zhou_fin_time_non_exp}, Mean-Variance problems, see e.g.\ \cite{Bayraktar_Zhang_Zhou,ChristensenLindensjoe_meanvariance} or reward functionals based on payoff functions with explicit dependence on the initial state, see \cite{ChristensenLindensjoe_equilibrium}. \cite{ChristensenLindensjoe_optimaldividend} considers optimal dividend problems where, similar to the problem we consider in the present work, the time-inconsistency stems from a moment constraint.
 
Especially for continuous time problems with non discrete state space $\mathcal{X}\subset \R^n$, there are multiple equilibrium concepts to model the previously mentioned aspect of consistent planing in optimal stopping. A notable aspect of such equilibrium problems is the choice of the set of admissible stopping times. Usually, to achieve compatibility with the common equilibrium concepts, some \textit{Markovian} structure must be imposed for stopping times to be admissible strategies. For Itô diffusions the monograph \cite{bjoerk2021time} introduces a version of the most common notion of an equilibrium in continuous time optimal stopping that is, like the one we consider, based on a first order condition, but with first exit times of the space time process as admissible stopping times. Other works that feature closely related equilibrium concepts based on a first order condition are \cite{BodnariuChristensenLindensjoe2022, ChristensenLindensjoe_mixedstrategy, ebert2020_weighteddiscount, zhou_fin_time_non_exp}. The first two even allow for mixed/randomized \textit{Markovian} stopping strategies. Mixed stopping strategies in discrete time frameworks are treated in \cite{Bayraktar_Zhang_Zhou, ChristensenLindensjoe_meanvariance}. \cite{bayraktar2022equilibria} introduces and compares the concepts of mild, weak, strong and optimal mild equilibria in a diffusion setting with only pure first entry times being admissible, while the weak equilibrium corresponds to the concept in \cite{bjoerk2021time}.

The introduction of randomized stopping times follows the well-known principle from game theory: In general, there are no equilibria in pure strategies. This is also the case for our problem class, see Example \ref{nofe}. This changes if the players are allowed to randomize. Since the randomization here must be of Markov type by the problem formulation (see \eqref{shiftprop} below), this leads to the class formally introduced in\ Definition \ref{rmt} below. For a detailed discussion and motivation we refer to \cite{BodnariuChristensenLindensjoe2022, ChristensenLindensjoe_mixedstrategy}.
\\

%%%%%%%% Expectation constraint
For the formulation of the equilibrium problem with expectation constraint we drew inspiration from corresponding (precommitment) optimal stopping problems. Such optimal stopping problems were first examined by Kennedy \cite{kennedy1982} for discrete time processes. Using Lagrangian techniques he reduces the constrained problem to a classical unconstrained one and obtains optimal stopping times in the constrained problem. 
The  Lagrangian approach has often been successfully employed to fully solve optimal stopping problems even in continuous time with expectation constraints, e.g.\ in \cite{Makasu, PeskirQuickestDetection2012, PedersenPeskirMeanVariance}.
%Makasu \cite{Makasu} derives an upper bound for the value of an optimal stopping problem with expectation constraint, where the payoff is determined by a geometric Brownian motion and a coupled diffusion process.
%Quickest detection problems with  expectation constraints are studied in Section 4 of \cite{PeskirQuickestDetection2012}, where the expected positive or negative deviation of the stopping time from the time to detect is assumed to be bounded.

%Horiguchi \cite{horiguchi2001} deals with optimally stopping a finite state process that also can be controlled with finitely many actions. Mathematical programming techniques allow to derive optimal stopping times satisfying an expectation constraint.

In \cite{miller2017} and \cite{AKK} the optimal stopping problem with expectation constraint is turned into an unconstrained optimal control problem with extended state space. 
The first article~\cite{miller2017} exploits the optimal stopping problem with expectation constraint for solving a time-inconsistent, but unconstrained stopping problem. The article \cite{AKK} formulates a dynamic programming principle (DPP), characterizes the value function in terms of a Hamilton-Jacobi-Bellman equation and proves a verification theorem. Bayraktar and Yao \cite{BayraktarYao2020} provide a proof of the DPP in a general non-Markovian framework with a series of inequality-type and equality-type expectation constraints. Moreover, \cite{BayraktarYao2023} extends \cite{BayraktarYao2020} to the case where in the constrained optimal stopping problem also the diffusion is controlled. 

In \cite{ChowYuZhou} a DPP for an optimal control problem with intermediate expectation constraints at each point in time in a general non-Markovian framework is derived. 

For optimal stopping problems of one-dimensional processes with stopping times satisfying an expectation constraint \cite{3points} shows that the set of stopping times can be restricted to those stopping times such that the law of the process at the stopping time is a weighted sum of three Dirac measures.
%In \cite{AnkirchnerKlein} a stopping problem with expectation constraint arising from a sequential testing of two simple hypotheses $H_0$ and $H_1$ on the drift rate of a Brownian motion is analyzed. For the optimal stopping time it is sufficient to consider rules following type: stop if the posterior probability for $H_1$ leave a given interval or stop if the posterior probability comes back to a fixed intermediate point after a sufficiently large excursion which means to abandon the observation without accepting $H_0$ or $H_1$.

%All articles dealing with expectation constraints mentioned above consider the maximization problem in the precommitment sense. 

\section{Model and problem formulation}\label{Ch2}

In this paper we consider a 1-dimensional Itô diffusion $X = (X_t)_{t \in [0,\infty)}$ taking values in an interval $(\alpha, \beta)$ with $ - \infty \le \alpha < \beta \le \infty$ and 
defined on a filtered probability space $(\Omega, \f, \mathbb F,\PR)$ with filtration $\mathbb F:= (\f_t)_{t \in [0,\infty)}$ satisfying the usual hypotheses. 
Generally we assume that $X$ is the strong solution to the stochastic differential equation
\begin{align}
    d X_t = \mu(X_t) dt + \sigma(X_t) dW_t, \quad X_0=x \notag
\end{align} with an $\mathbb F$-adapted, real valued, standard Brownian motion $W = (W_t)_{t \in [0,\infty)}$ and Lipschitz continuous coefficients $\mu\colon (\alpha, \beta) \to \R$, $\sigma\colon (\alpha, \beta) \to (0,\infty)$. The infinitesimal generator $\mathcal{A}$ of $X$ is given by $\mathcal{A}f(x)=\mu(x)\frac{\partial f}{\partial x}(x) +\frac{\sigma^2(x)}{2} \frac{\partial^2 f}{\partial x^2}(x)$ with $f:(\alpha,\beta)\to \R$ twice continuously differentiable. 

Other general notations used in the following are $\overline{D}$ for the closure of a set $D\subset (\alpha,\beta)$, $D^\circ$ for its interior, $\partial D$ for its boundary and $D^c$ for its complement $(\alpha, \beta)\setminus D$. Given $D_1\subset D_2\subset (\alpha,\beta)$ and a function $f\colon D_2\to \R$ we denote the restriction of $f$ to $D_1$ by $f\vert_{D_1}$. For $x\in(\alpha,\beta)$ and $f\colon(\alpha,\beta)\to \R$ we set $f(x-):=\lim_{y\nearrow x}f(y)$ and $f(x+):=\lim_{y\searrow x} f(y)$ whenever the corresponding limit exists.

Furthermore, let $\mathbb F^X:= (\f^{X}_t)_{t \in [0,\infty)}$ be the canonical filtration of $X$, i.e.\ $\f_t^X= \sigma(X_s: s\le t)$. We also set $\f^X_\infty := \sigma( \bigcup_{t \ge 0}\f_t^X) $. 
For every Borel measurable $D \subset (\alpha, \beta)$ and $h>0$ we define $\mathbb F^X$-stopping times
\begin{align}
    \tau^D &:= \inf \{ t \ge 0 \colon X_t \not\in D\} \quad \ttt{ and}\notag\\ 
    \tau_h&:= \inf \{ t\ge 0 \colon \vert X_t - X_0 \vert \ge h\}.\notag
\end{align}

Moreover, we assume that there is an $\mathbb F$-adapted Poisson process $N=(N_t)_{t \in [0,\infty)}$ on $(\Omega, \f, \mathbb F,\PR)$ with intensity 1, that is independent of $X$ and $W$. The process $N$ acts as an \textit{external} source of randomness in the model that is going to be introduced in the following.

For fixed $D\subset (\alpha,\beta)$ the set of càdlàg functions $\psi\colon(\alpha, \beta) \to [0, \infty)$ with finitely many discontinuities that vanish on $(\alpha, \beta) \setminus \overline{D}$ shall be denoted by $\mathfrak{V}_D$. For $\psi\in \mathfrak{V}_D$ we set
\begin{align}
    \Psi=(\Psi_t)_{t\in [0,\infty)}, \quad \Psi_t := \int_0^t \psi(X_s) ds\label{defpsi}
\end{align} 
and denote the set of such processes $\Psi$ by $\mathfrak{W}_D$. For each $D \subset (\alpha,\beta)$ and $\Psi \in \mathfrak{W}_D$ let $\mathbb F^{X,N_\Psi}:=(\f^{X,N_\Psi}_t)_{t\in [0,\infty)}$ given by $\f^{X,N_\Psi}_t:= \sigma((X_s,N_{\Psi_s}): s \le t)$ for all $t \in [0, \infty)$ denote the canonical filtration of the 2-dimensional Markov process $((X_t,N_{\Psi_t}))_{t\in [0,\infty)}$. Moreover, let $\f^{X,N_\Psi}_\infty:= \sigma(\bigcup_{t\ge 0}\f^{X,N_\Psi}_t)$. With this we define $\tau^\Psi$ as the first jump time of the process $N_\Psi$, i.e.
\begin{align*}
    \tau^\Psi := \inf\{t \ge 0\colon N_{\Psi_t}\ge1\}.
\end{align*} 
Moreover, we set
\begin{align}
    \tau^{D, \Psi} := \tau^D \wedge \tau^\Psi \label{randmark}
\end{align} where $\wedge$ denotes the minimum. 

As usual we denote the conditional distribution of $\PR$ given $(X_0,N_{\Psi_0}) = (x,n)$ by $\PR_{x,n}$ and the expectation with respect to $\PR_{x,n}$ by $\Ex_{x,n}$. For short we write $\PR_x$ and $\Ex_{x}$ instead of $\PR_{x,0}$ and $\Ex_{x,0}$, respectively.

\begin{definition} (randomized Markovian time) \label{rmt}\\
If $D$ is open, we call stopping times of type \eqref{randmark} \textit{randomized Markovian time}. 
The set of randomized Markovian times is denoted by
\begin{align}
    \mathcal{S} &:= \{ \tau^{D, \Psi} \colon D\subset(\alpha, \beta) \ttt{ open}, \Psi\in \mathfrak{W}_D \} \notag
\end{align} 
and define the set of randomized Markovian times with expectation up to $T\in [0,\infty)$ as
\begin{align}
    \mathcal{S}(T) := \{ \tau \in \mathcal{S}\colon \Ex_x [\tau] \le T \,\,\,\, \forall \, x\in (\alpha, \beta) \}.\notag
\end{align}
\end{definition}

\begin{remark}\label{rem1}
\begin{enumerate}[(i)]
    \item The requirement of right continuity and existence of real left limits in the definition of $\mathfrak{V}_D$ could also be replaced by left continuity and the existence of real right limits with all the following statements and proofs remaining analogous. As $\Psi$ defined by \eqref{defpsi}, and thus also any randomized Markovian time $\tau^{D,\Psi}$, depends on $\psi$ only up to null sets, the set $\mathcal{S}$ remains unchanged by this assumption anyway.
    % Das hier braucht man übrigens in \eqref{a1}
    \item The fact that the definition of $\mathcal{S}$ only involves open $D\subset (\alpha,\beta)$ instead of general Borel measurable $D\subset (\alpha,\beta)$ is not a loss of generality, since for every Borel measurable $D\subset (\alpha,\beta)$ and $\Psi\in \mathfrak W_D$ the randomized Markovian time $\tau^{D^\circ,\Psi}\in \mathcal{S}$ has the same distribution as $\tau^{D,\Psi}$. To show this note first that $\mathfrak W_D\subset \mathfrak W_{D^\circ}$. The Lipschitz continuity of $\mu$ and $\sigma$ and the assumption $\sigma >0$ ensure $\PR_x\!\left(\tau^D = \inf\{t\geq 0 \colon X_t \in \overline{D^c}\}=\tau^{D^\circ}\right)=1$ for each $x\in(\alpha,\beta)$, cf.\ \cite[Chapter V, Lemma (46.1)]{RogersWilliamsVol2}. This yields $\PR_x\!\left(\tau^{D,\Psi} = \tau^D\wedge \tau^\Psi=\tau^{D^\circ}\wedge\tau^\Psi=\tau^{D^\circ,\Psi}\right)=1$ for each $x\in(\alpha,\beta)$.
    %For a proof see Appendix \ref{stopsofort}. Thus in the definition of randomized Markovian times the additional requirement of $D$ being open is no actual restriction.
    \item The function $\psi$ can be regarded as the \textit{rate of randomized stopping} of the stopping times $\tau^\Psi$ and $\tau^{D,\Psi}$ respectively.
    \item For each randomized Markovian time $\tau^{D,\Psi}$ there is an associated partition of $(\alpha, \beta)$ into 
    \begin{itemize}
        \item the \textit{stopping region} $D^c$, 
        \item the \textit{continuation region} $D\cap \{x\in (\alpha,\beta)\colon \psi(x)=0\}$ as well as 
        \item the \textit{set of randomized stopping} or \textit{randomization region} 
        $D\setminus \{x\in (\alpha,\beta)\colon $ $\psi(x)=0\}.$
        \end{itemize}
    \item\label{rem15} If we define $U:=\inf \{t \ge 0\colon  N_t\ge1\}$ as the first jump time of $N$, we have $\tau^\Psi= \inf \{t \ge0\colon  \Psi_t \ge U\}$. Moreover, $U \sim \Exp(1)$ is independent of $X$.

    \item Keep in mind that all $\tau \in \mathcal{S}$ share the same \textit{external source of randomization} $N$ that also defines $U$.
    \item \label{conddistr} 
    The structure of the stopping times $\tau^\Psi$ enables a simple calculation of the corresponding expectation functionals: Let $x\in (\alpha,\beta)$. Given any random variable $Y$ on $(\Omega,\f,\PR_x)$ with values in a measure space $(E, \mathcal{E})$ we denote the image measure of $Y$ on $(E, \mathcal{E})$ by $\PR_x^Y$. Let $\pi_t\colon (\alpha,\beta)^{[0,\infty)}\to (\alpha,\beta)$, $x\mapsto x(t)$ denote the projection on the $t$-coordinate. We regard $X$ as a random variable with values in $((\alpha,\beta)^{[0,\infty)},$ $\bigotimes_{t\in [0,\infty)}\mathcal{B}((\alpha,\beta)^{[0,\infty)}))$, where $\bigotimes_{t\in [0,\infty)}$ denotes the product of sigma algebras of the indexed set and $\mathcal{B}(E)$ the Borel sigma algebra on a topological space $E$.
    
    Let $A:= \{\pi_{t_1}^{-1}(\Gamma_1)\cap\ldots \cap\pi_{t_n}^{-1}(\Gamma_n)\}$, $t,t_1,...,t_n\in [0,\infty)$, $n\in \N$ and $\Gamma_1,...,\Gamma_n\subset (\alpha,\beta)$ be Borel measurable. By the independence of $X,U$ using $U\sim \Exp(1)$
    \begin{align}
        &\Ex_x[\ind{X\in A}\ind{\tau^\Psi\le t}] \notag
        =\Ex_x[\ind{X\in A}\ind{\int_0^t\psi(X_s)ds \ge U}]\notag\\[0.2cm]
        =&\int_{(\alpha,\beta)^{[0,\infty)}} \ind{y\in A} \int_{[0,\infty)}  \ind{\int_0^t\psi(\pi_s(y))ds \ge u} \PR^U(du) \PR_x^X(dy)\notag\\[0.2cm]
        =& \int_{(\alpha,\beta)^{[0,\infty)}} \ind{y\in A} \left(1-e^{-\int_0^t\psi(\pi_s(y))ds}\right) \PR_x^X(dy)\notag\\[0.2cm]
        =&\Ex_x\left[\ind{X\in A} \left(1-e^{-\int_0^t\psi(\pi_s(X))ds}\right)\right]
        = \Ex_x\left[\int_0^t \ind{X\in A}\psi(X_s)e^{-\int_0^s\psi(X_r)dr}ds \right].\notag
    \end{align} Since $\big(\bigotimes_{t\in [0,\infty)}\mathcal{B}((\alpha,\beta)^{[0,\infty)})\big)\otimes \mathcal{B}([0,\infty)) = \sigma\big(\{\pi_{t_1}^{-1}(\Gamma_1)\cap\ldots \cap\pi_{t_n}^{-1}(\Gamma_n)\cap [0,t]:t,t_1,...,t_n\in [0,\infty)$, $n\in \N$, $\Gamma_1,...,\Gamma_n\in \mathcal{B}((\alpha,\beta)^{[0,\infty)})\}\big)$ with $\otimes$ denoting the product sigma algebra, this extends to $$\Ex_x[f(X,\tau^\Psi)]=\Ex_x\left[\int_{[0,\infty)}f(X,s)\psi(X_s) e^{-\int_0^s\psi(X_r)dr}ds\right]$$ for all integrable $f\colon(\alpha,\beta)^{[0,\infty)}\times [0,\infty)\to \R$. In particular $$\Ex_x[f(X,\tau^\Psi)\vert \f^X_\infty]=\int_{[0,\infty)}f(X,s)\psi(X_s)e^{-\int_0^s\psi(X_r)dr}ds.$$
    \item $\tau^{D, \Psi}$ is the first entry time of the right continuous process $(X_t,N_{\Psi_t})_{t \in [0,\infty)}$ into the closed set $(D^c \times [0,\infty)) \cup ( (\alpha,\beta) \times [1,\infty))$ and thus an $\mathbb F^{X,N_{\Psi}}$-stopping time. 
    %For open $D\subset (\alpha, \beta)$, as a first entry time of $X$ in a closed set, $\tau^D$ is a $(\f_t^X)_{t\in[0,\infty)}$-stopping time. As $\Psi_t$ depends continuously on $X$ up to time $t$ and the filtration $(\f_t^{X,N})_{t\in[0,\infty)}$ includes all the information about $N$, the mapping $(t,\omega) \mapsto N_{\Psi_t}(\omega)-N_{\Psi_0}(\omega)$ is $\f_t^{X,N}$-measurable in $\omega$ and right continuous in $t$. $\tau^\Psi$ is the first entry time of $N_{\Psi_t}-N_{\Psi_0}$ into the closed set $[1,\infty)$, thus it is a $(\f_t^{X,N})_{t\in[0,\infty)}$-stopping time.
\end{enumerate}

\end{remark}

From Remark \ref{rem1} \eqref{rem15} we infer that $\inf\{t\ge 0\colon \Psi_t\ge \tilde U\}$ with $\tilde U\sim \Exp(1)$ independent of $X$ has the same distribution as $\tau^\Psi$. This raises the question why we use an entire process $N$ as the external source of randomness for the randomized Markovian times instead of a single random variable $\tilde U$. The reason is that in the subsequent analysis we need the randomized Markovian  times $\tau$ to have a Markov property of the form
\begin{align}
    \ind{\tau \ge \sigma}  \tau = \ind{\tau \ge \sigma} (\theta_\sigma \circ \tau + \sigma)\label{shiftprop}
\end{align} in distribution for all $\mathbb F^X$-stopping times $\sigma$ and some shift operator $\theta$. Choosing $\theta$ as a natural shift operator of a strong Markov process has the advantage of circumventing to debate well-definedness.

The problem is that it is not clear how to reasonably regard $\inf\{t\ge 0: \Psi_t\ge \tilde U\}$ as a stopping time with respect to a strong Markov process. Technically it would work to add information about $\tilde U$ to $\mathbb F^X$ at all times but this busts the interpretation of the problem as it provides the agent with the information on which path randomization leads to earlier or later stopping. Thus instead of a canonical shift of a Markov process we would need to artificially extend the shift of $X$ to non-$\mathbb F^X$-measurable functions like $\inf\{t\ge 0: \Psi_t\ge \tilde U\}$.

For fixed open $D\subset (\alpha,\beta)$ and $\psi \in \mathfrak{V}_D$ defining $\Psi \in \mathfrak{W}_D$ via \eqref{defpsi} the strong Markov process $(X,N_\Psi)$ induces a natural shift operator that can either be defined on the space $\mathcal{H}=\mathcal{H}_{D,\Psi}$ of $\f^{X,N_\Psi}_\infty$-measurable functions, cf. \cite[p.\ 115]{oksendal_5th} or directly on the underlying probability space, cf. \cite[Chapter I, Sections 3, 7]{blumenthal_getoor}. In any case $\theta_\sigma \circ \tau$ is well defined as $\sigma$ is an $\mathbb F^{X,N_\Psi}$-stopping time. Going with the shift from \cite[p.\ 115]{oksendal_5th} the stopping time $\tau$ regarded as an element of $\mathcal{H}$ is not in the dense subset of $\mathcal{H}$ where the shift is initially defined, so technically there is the need to calculate how $\theta_\sigma$ extends to $\tau^\Psi$. With the procedure outlined in \cite[p.\ 115]{oksendal_5th} we then obtain
\begin{align*}
    \theta_\sigma \circ X_t &= X_{t+\sigma},\\[0.1cm]
    \theta_\sigma \circ \tau^D &= \inf \{ t\ge 0 \colon \theta_\sigma \circ X_{t}\not \in D\} =\inf \{ t\ge 0 \colon  X_{t+\sigma}\not \in D\},\\[0.1cm]
    \theta_\sigma \circ \tau^\Psi &= \inf \{ t \ge 0\colon  N_{\Psi_{\sigma+t}} \ge 1\}, \\[0.1cm]
    \theta_\sigma \circ \tau^{D, \Psi} &= \theta_\sigma\circ \tau^D \wedge \theta_\sigma\circ \tau^\Psi.
\end{align*}This yields the desired property \eqref{shiftprop}.

Next we introduce a certain way of concatenating randomized Markovian times that will be crucial for the definition of the (weak) equilibrium that is going to follow.

\newpage
\begin{definition}\label{lpo}(local perturbation operator)\\
For $h \ge0$ and randomized Markovian times $\tau^{(1)},\tau^{(2)} \in \mathcal{S}$ we define the local perturbation operator $\diamond$ of $\tau^{(1)}$ with $\tau^{(2)}$ on a radius of $h$ by
\begin{align}    
    \tau^{(1)} \diamond \tau^{(2)}(h) := \ind{\tau^{(2)} \le \tau_h} \tau^{(2)} + \ind{\tau^{(2)} > \tau_h} ( \theta_{\tau_h} \circ \tau^{(1)}+ \tau_h). \notag
\end{align} We set $\mathcal{S}^\diamond:=\{ \tau^{(1)} \diamond \tau^{(2)}(h): \tau^{(1)},\tau^{(2)}\in \mathcal{S}, h\ge 0\}$.
\end{definition}

The name local perturbation operator comes from the fact that $\tau^{(1)} \diamond \tau^{(2)}(h)$ equals $\tau^{(2)}$ up to  time $\tau_h$ where $X_t$ exits $[X_0-h,X_0+h]$ for the first time and then instead goes by the rule of $\tau^{(1)}$.

\begin{remark} \label{t5} Let $D,\tilde D\subset (\alpha,\beta)$ be open, $\Psi\in \mathfrak{W}_D$, $\tilde \Psi\in \mathfrak W_{\tilde D}$ and $h\ge 0$.
\begin{enumerate}[(i)]
    \item \label{t51} For $\tau^{D,\Psi}$,  $\Psi\in \mathfrak{W}_D$ and all $h \ge 0$ by \eqref{shiftprop} we infer 
    \begin{align}
    \tau^{D,\Psi} = \tau^{D,\Psi} \diamond \tau^{D,\Psi}(h). \notag
    \end{align} Thus in particular $\mathcal{S}\subset \mathcal{S}^\diamond$.
    \item Generally $\tau^{D,\Psi}\diamond \tau^{\tilde D, \tilde \Psi}(h)$ is neither an $\mathbb F^{X,N_\Psi} $- nor an $\mathbb F^{X,N_{\tilde \Psi}}$-stopping time, since it draws information from $N_\Psi$ and $N_{\tilde \Psi}$. Due to the path dependence induced by the indicators $\ind{\tau^{\tilde D, \tilde \Psi}\le \tau_h}$, $\ind{\tau^{\tilde D, \tilde \Psi} > \tau_h}$ the local perturbation $\tau^{D,\Psi}\diamond \tau^{\tilde D, \tilde \Psi}(h)$ is not in $\mathcal{S}$ and thus $\mathcal{S}\subsetneq\mathcal{S}^\diamond$.
\end{enumerate}
\end{remark}

\begin{definition}\label{equil} (reward functional) \\
For all almost surely finite $\tau\in \mathcal{S}^\diamond$ we define the reward functional
\begin{align}
    J_\tau(x) := \Ex_x [e^{-r\tau} g(X_\tau)], \,\, x\in (\alpha, \beta) \notag
\end{align} with some fixed discount factor $r \ge 0$ and a measurable payoff function $g \colon (\alpha, \beta)\to[0,\infty)$.
\end{definition}
\begin{definition} \label{equil2}(equilibrium)\\
Let $T\in[0,\infty)$ be a fixed constant. We call a stopping time $\tau^* \in \mathcal{S}(T)$ equilibrium randomized Markovian time for the optimal stopping problem associated to the functional $J$ with expectation constraint $T$, or just equilibrium for brevity, if we have
\begin{align}
    \liminf_{h \searrow 0} \frac{J_{\tau^*}(x)- J_{\tau^* \diamond \tau (h)}(x)}{\Ex_x[\tau_h]} \ge 0 \label{equi}
\end{align} for all $x\in(\alpha, \beta)$ and any $\tau \in \mathcal{S}$ such that there exists some $h>0$ and a neighborhood $U$ of $x$ with $\Ex_y[\tau^* \diamond \tau (h')]\le T$ for all $h'\in[0,h)$ and all $y \in U$.
\end{definition}
Note that the additional condition on the expectation(s) of $\tau^* \diamond \tau (h')$ excludes deviating strategies from consideration that locally violate the constraint.

In order to investigate the fulfillment of the constraint for a given randomized Markovian time and shorten the notation the following definition will be useful.
\begin{definition} (expected time function)\\
Let $\tau^{D, \Psi} \in \mathcal{S}$ be a randomized Markovian time. We call the function 
\begin{align}
    \edp(x):= \Ex_x[\tau^{D,\Psi}]\notag
\end{align} expected time function.
\end{definition}

\section{Equilibrium randomized Markovian times} \label{Ch3}
The aim of this section is to derive some necessary conditions for equilibrium strategies in the sense of Definition \ref{equil2}. Moreover, these conditions will help us to identify \textit{natural} parts of the Verification Theorem \ref{veri} in Section \ref{subsec:verification}.

\subsection{Properties of equilibria}
We now provide  necessary conditions for an equilibrium.

\begin{proposition}\label{propeq}
Let $\tau^{D, \Psi} \in \mathcal{S}(T)$ be a randomized Markovian time and $\psi \in \mathfrak{V}_D$ the function defining $\Psi$ via \eqref{defpsi}. Assume that $g $ is right continuous and has real left limits. Moreover, let $J_{\tau^{D,\Psi}}\vert_D$ be continuous and let $\tau^{D,\Psi}$ be an equilibrium in the sense of Definition \ref{equil2}. We have
\begin{enumerate}[(i)]
    \item $(\mathcal{A}-r)g(x)\le 0$ for all $x \in (D^c)^\circ$ whenever $g\vert_{(D^c)^\circ}\in \ccc^2((D^c)^\circ)$,\label{a}
    \item $g(x) \le J_{\tau^{D,\Psi}}(x)$ for all $x \in(\alpha, \beta),$ \label{b}
    \item \label{d} if $x\in \partial D$ such that there exists some $h >0$ with $ J_{\tau^{D,\Psi}}\vert_{(x-h,x]} \in \ccc^2((x-h,x])$ as well as $ J_{\tau^{D,\Psi}}\vert_{[x,x+h)} \in \ccc^2([x,x+h))$, then $\partial_x J_{\tau^{D,\Psi}}(x+) \le \partial_x J_{\tau^{D,\Psi}}(x-)$,
    \item $(\spt(\psi))^\circ \subset \{x \in (\alpha, \beta)\colon g(x) = J_{\tau^{D,\Psi}}(x) \text { or } e_{D,\Psi}(x)=T \}$, where $\spt(\psi)$ denotes the support of $\psi$. \label{c}
\end{enumerate}
\end{proposition}

\begin{remark}
The parts \eqref{a} and \eqref{b} in Proposition \ref{propeq} are standard conditions for optimal stopping times.\ \eqref{a} means that the drift of the stopped process needs to be non-positive in the stopping region and \eqref{b} says that any equilibrium payoff needs to be at least as high as the payoff for stopping immediately. 

\eqref{d} can be regarded as a generalized smooth fit condition as in \cite[Formula (24)]{bayraktar2022equilibria} and  \cite[Condition (IV) in Theorem 8]{BodnariuChristensenLindensjoe2022}. Observe that in \cite{BodnariuChristensenLindensjoe2022} the smooth fit condition holds at the boundary between continuation and randomization region, whereas our generalized smooth fit condition has to be satisfied at the boundary between continuation and stopping region.

\eqref{c} is the most interesting part of this proposition. In other words it reads \textit{in an equilibrium randomization can only be used if we are indifferent between stopping and continuing or to fully exploit the expected time constraint before stopping}.

In Lemma \eqref{Jstetigkeit} in the appendix we derive a simple sufficient condition for the required continuity of $J_{\tau^{D,\Psi}}\vert_D$.
\end{remark}

\begin{proof}[Proof of Proposition \ref{propeq}]
We start with \eqref{a}. Let $x \in (D^c)^\circ$ and $\tau \in \mathcal{S}$ such that the requirements from Definition \ref{equil2} are met. The definition of the perturbation operator yields
    \begin{align}
        \tau^{D,\Psi} \diamond \tau(h) &= \ind{\tau \le \tau_h} \tau + \ind{\tau > \tau_h} ( \theta_{\tau_h}  \circ\tau^{D,\Psi} +\tau_h) = \ind{\tau \le \tau_h} \tau + \ind{\tau > \tau_h} \tau_h= \tau \wedge \tau_h, \label{stefandanke}
    \end{align} $\PR_x$-a.s. for all $h\ge0$ such that $[x-h,x+h]\subset D^c$. Here we took into account that $\PR_x(\theta_{\tau_h} \circ\tau^{D,\Psi}  =0)=1$ for such $h$. With \eqref{stefandanke} we get
\begin{align}
     0 \le &\liminf_{h \searrow 0} \frac{J_{\tau^{D,\Psi}}(x) - J_{\tau^{D,\Psi} \diamond \tau(h)}}{\Ex_x[\tau_h]} = \liminf_{h \searrow 0} \frac{J_{\tau^{D,\Psi}}(x) - J_{\tau \wedge \tau_h}(x)}{\Ex_x[\tau_h]} \notag\\[0.2cm]
        =& \liminf_{h \searrow 0}  \frac{1}{\Ex_x[\tau_h]} \big( g(x) - \Ex_x\big[e^{-r\tau \wedge \tau_h}g(X_{\tau\wedge \tau_h})\big] \big)\notag\\[0.2cm]
        =& \liminf_{h \searrow 0} \frac{\Ex_x[\tau \wedge \tau_h]}{\Ex_x[\tau_h]}\frac{1}{\Ex_x[\tau \wedge \tau_h]}\bigg( -\Ex_x\bigg[\int_0^{\tau \wedge \tau_h}(\mathcal{A}-r) e^{-rs}g(X_s)ds \bigg] \bigg)\notag\\[0.2cm]
        =& (r-\mathcal{A})g(x),\label{glre}
\end{align} where we use that $\lim_{h \to 0}\frac{\Ex_x[\tau \wedge \tau_h]}{\Ex_x[\tau_h]}=1$, cf.\ \cite[Lemma 26]{BodnariuChristensenLindensjoe2022} in the last line.

The proof of \eqref{b} is rather obvious. If \eqref{b} would not hold at some point $x\in (\alpha, \beta)$, choosing $\tau:\equiv0$, which corresponds to $\tau^{\varnothing, 0}$ in our notion of randomized Markovian times, we would end up with a contradiction as follows:
\begin{align*}
   0 \le \liminf_{h \searrow 0} \frac{J_{\tau^{D,\Psi}}(x) - J_{\tau^{D,\Psi} \diamond \tau(h)}}{\Ex_x[\tau_h]} = \liminf_{h \searrow 0} \frac{J_{\tau^{D,\Psi}}(x) - g(x)}{\Ex_x[\tau_h]} = - \infty.
\end{align*}

In order to prove \eqref{c} we show that if 
\begin{align*}
    x\in (\alpha, \beta) \setminus \{x \in (\alpha, \beta)\colon g(x) = J_{\tau^{D,\Psi}}(x) \text{ or } e_{D,\Psi}(x)=T \},
\end{align*}
then $x \in (\alpha, \beta) \setminus (\spt(\psi))^\circ$. First observe that $g = J_{\tau^{D,\Psi}}$ on $ D^c$, hence $x \in D$. $\tau^{D,\Psi}$ being an equilibrium implies $e_{D,\Psi}\le T$ which yields $e_{D,\Psi}(x)<T$. Moreover, from \eqref{b} we read off $g(x)<J_{\tau^{D,\Psi}}(x)$. For $R>0$ such that $\tilde{D}=(x-R,x+R)\subset (\alpha,\beta)$ and $\tilde{\Psi}:\equiv0$ set $\tau^R:=\tau^{\tilde{D},\tilde{\Psi}}$. By Definition \ref{lpo}, the Markov property and Lemma \ref{rightcont} it holds that
\begin{align*}
    \Ex_y[\tau^{D,\Psi} \diamond \tau^R(h)] &= \Ex_y[\ind{\tau^R \le \tau_h}\tau^R] + \Ex_y[\ind{\tau^R > \tau_h} e_{D,\Psi}(X_{\tau_h})] + \Ex_y[\ind{\tau^R > \tau_h} \tau_h]\\[0.1cm]
    &\le \Ex_y[\tau_h] + \Ex_y[\ind{\tau^R > \tau_h} e_{D,\Psi}(X_{\tau_h})] \stackrel{h \searrow 0}{\to} \inda{\tilde{D}}(y) \edp(y)\stackrel{y\to x}{\to}\edp(x)<T.
\end{align*} Thus for $R,h$ sufficiently small we obtain $\Ex_y[\tau^{D,\Psi} \diamond \tau^R(h)]\le T$ for all $y\in(\alpha,\beta)$.  

Hence by Definition \ref{equil2} we have
\begin{align}
\liminf_{h \searrow 0} \frac{J_{\tau^{D,\Psi}}(x)- J_{\tau^{D,\Psi} \diamond \tau^R(h)}(x)}{\Ex_x[\tau_h]} \ge 0. \notag
\end{align}
For $h< R$ it holds that
\begin{align*}
    \tau^{D,\Psi} \diamond \tau^R(h) = \ind{\tau^R \le \tau_h} \tau^R + \ind{\tau^R > \tau_h}(  \theta_{\tau_h}\circ \tau^{D,\Psi} + \tau_h) =   \theta_{\tau_h} \circ \tau^{D,\Psi} + \tau_h, \,\,\,\, \PR_x\ttt{-a.s.}
\end{align*} As $x \in D$ we can apply this together with Lemma \ref{Kri10} and Remark \ref{t5} \eqref{t51}, which leads to 
\begin{align*}
    0 \le&\,\liminf_{h \searrow 0} \frac{J_{\tau^{D,\Psi}}(x)- J_{\tau^{D,\Psi} \diamond \tau^R(h)}(x)}{\Ex_x[\tau_h]} = \liminf_{h \searrow 0} \frac{J_{\tau^{D,\Psi}\diamond \tau^{D,\Psi}(h)}(x)- J_{ \theta_{\tau_h} \circ\tau^{D,\Psi}  + \tau_h}}{\Ex_x[\tau_h]}\\
    =&\, \frac{1}{2}\psi(x-)( g(x-)-J_{\tau^{D,\Psi}}(x) )+\frac{1}{2} \psi(x)(g(x)-J_{\tau^{D,\Psi}}(x)  )\\
    \le&\, \frac{1}{2} \psi(x)(g(x)-J_{\tau^{D,\Psi}}(x)  ),
\end{align*} so $\psi(x)=0$ needs to hold because $g(x)-J_{\tau^{D,\Psi}}(x)< 0$. This means $x \not \in(\spt(\psi))^\circ$, which was the claim of \eqref{c}.

We finish by showing \eqref{d} in a very similar way. Consider $x \in \partial D$ and once again $R>h>0$. First observe that $\edp(x)=0$, so just as before $\Ex_y[\tau^{D,\Psi} \diamond \tau^R(h)]\le T$ for all $y\in(\alpha,\beta)$ and sufficiently small $h$ and $R$. By \eqref{aux2} and the local-space-time formula, cf.\ \cite{Peskir2007}, with $L^x$ denoting the local time of $X$ at $x$, we obtain
\begin{align*}
    &\frac{J_{\tau^{D,\Psi}}(x)- J_{\tau^{D,\Psi} \diamond \tau^R(h)}(x)}{\Ex_x[\tau_h]} 
    %=& \frac{J_{\tau^{D,\Psi}}(x) - \Ex_{x}[ \ind{\tau^R \le \tau_h } e^{-r \tau^R} g(X_{\tau^R}) + \ind{\tau^R > \tau_h}e^{-r \tau_h} J_{\tau^{D,\Psi}}(X_{\tau_h})] }{\Ex_x[\tau_h]} \\
    =\frac{J_{\tau^{D,\Psi}}(x)- \Ex_{x}[ e^{-r \tau_h} J_{\tau^{D,\Psi}}(X_{\tau_h})] }{\Ex_x[\tau_h]}\\[0.1cm]
    =&-\frac{\Ex_x \left[\int_0^{\tau_h}e^{-rs} (\mathcal{A}-r) J_{\tau^{D,\Psi}} (X_s)ds + \frac{1}{2} \int_0^{\tau_h} \partial_x J_{\tau^{D,\Psi}} (x+) - \partial_x J_{\tau^{D,\Psi}}(x-)dL_s^x \right]}{\Ex_x[\tau_h]}\\[0.1cm]
    =& -\frac{\Ex_x \left[\int_0^{\tau_h} \ind{X_s \neq x} e^{-rs} (\mathcal{A}-r) J_{\tau^{D,\Psi}}(X_s) ds \right]}{\Ex_x[\tau_h]} - \frac{1}{2} (\partial_x J_{\tau^{D,\Psi}} (x+) - \partial_x J_{\tau^{D,\Psi}}(x-))\frac{   \Ex_x[L_{\tau_h}^x ]}{\Ex_x[\tau_h]}.
\end{align*} 
Due to the $C^2$ assumption on $J_{\tau^{D,\Psi}}$ the first summand is bounded for $h \searrow 0$. Note that by \cite[Proposition 3.3]{ChristensenLindensjoe_mixedstrategy} we have $\frac{   \Ex_x[L_{\tau_h}^x ]}{\Ex_x[\tau_h]} \to \infty$ for $h\searrow 0$. Thus, if we assume $\partial_x J_{\tau^{D,\Psi}} (x+) - \partial_x J_{\tau^{D,\Psi}}(x-)>0$ we get
\begin{align*}
   \liminf_{h \searrow 0 }\frac{J_{\tau^{D,\Psi}}(x)- J_{\tau^{D,\Psi} \diamond \tau^R(h)}(x)}{\Ex_x[\tau_h]} = -\infty.
\end{align*} This contradicts the equilibrium condition $\liminf_{h \searrow 0 }\frac{J_{\tau^{D,\Psi}}(x)- J_{\tau^{D,\Psi} \diamond \tau^R(h)}(x)}{\Ex_x[\tau_h]} \ge 0$.
\end{proof}

\subsection{Verification Theorem}\label{subsec:verification}
In this subsection we derive a Verification Theorem that provides sufficient conditions for a candidate randomized Markovian time $\tau^{D,\Psi}\in \mathcal{S}$ to be an equilibrium.  
We start by introducing a notions of regularity for the continuation region $D$ of the candidate $\tau^{D,\Psi}$ and the payoff function $g$.

\begin{definition}\label{regular}
    An open $D\subset (\alpha, \beta)$ is called regular if for every $x\in\partial D$ one of the following conditions holds true:
    \begin{itemize}
    \item $x$ is an isolated point in $D^c$.
    \item There exists $\eps>0$ such that either $(x-\eps,x) \subset D, [x,x+\eps)\subset D^c$ or $(x-\eps,x]\subset D^c, (x,x+\eps)\subset D$.
\end{itemize}
Moreover, we call the payoff function $g\colon (\alpha,\beta) \to [0,\infty)$ regular with respect to $D$ if $g$ is $\ccc^2$ on $(D^c)^\circ$ and $g(x+), g(x-) $ exist in $\R$ for all $x\in D$.
\end{definition}

This notion of regularity of the stopping region can also be found in \cite{bayraktar2022equilibria} and excludes the case from consideration where $D$ consists of infinitely many disjoint intervals that cluster towards some points as this would pose technical difficulties. 

\begin{lemma}\label{Kri16}
Let $\tau^{D, \Psi} \in \mathcal{S}$ be a randomized Markovian time with regular $D$, $\psi\in \mathfrak{V}_D$ the function defining $\Psi$ via \eqref{defpsi} and $x \in D$. Assume that $\psi$ is continuous in $x$, $g(x+),g(x-)\in \R$ exist and that there exists $\eps>0$ such that $J_{\tau^{D,\Psi}}\vert_{(x-\eps,x+\eps)} \in \ccc^2((x-\eps,x+\eps))$. Then we have
\begin{align}
    (\mathcal{A}-r)J_{\tau^{D,\Psi}}(x)= \psi(x)\left(J_{\tau^{D,\Psi}}(x)- \frac{1}{2}(g(x-)+g(x+))\right). \notag
\end{align}
\end{lemma}

\begin{beweis}
By Lemma \ref{Kri10} together with $\tau^{D,\Psi}= \tau^{D,\Psi} \diamond \tau^{D,\Psi}(t)$, cf.\  Remark \ref{t5} \eqref{t51}, for all $t \ge 0$ as well as the continuity of $\psi$ in $x$ we get
\begin{align}
    \lim_{h \searrow 0} \frac{ J_{\theta_{\tau_h} \circ \tau^{D, \Psi} +\tau_h}(x) -J_{\tau^{D,\Psi}} (x) }{\Ex_x[\tau_h]}&= \psi(x)\left(J_{\tau^{D,\Psi}}(x)- \frac{1}{2}\left(g(x-) +g(x+) \right)\right).\label{nun}
\end{align}
Consider $h \in (0, \eps)$ such that $(x-h,x+h) \subset D$. Using \eqref{aux1} followed by Dynkin's formula, cf. \cite[p.\ 133]{dynkin_markov} one obtains
    %&\lim_{h \searrow 0} \frac{ J_{\tau^{D, \Psi} \circ \theta_{\tau_h}(x) +\tau_h} -J_{\tau^{D,\Psi}} (x) }{\Ex_x[\tau_h]} = \lim_{h \searrow 0} \frac{ \Ex_{x}[e^{-r\tau_h} J_{\tau^{D,\Psi}}(X_{\tau_h}) ] -J_{\tau^{D,\Psi}} (x) }{\Ex_x[\tau_h]}\\
    %=& \lim_{h \searrow 0} \frac{ e^{-r0}J_{\tau^{D,\Psi}}(x) + \Ex_x\left[\int_0^{\tau_h} \left(\mathcal{A} +\frac{\partial}{\partial s}\right)e^{-rs} J_{\tau^{D,\Psi}}(X_s) ds \right] -J_{\tau^{D,\Psi}} (x) }{\Ex_x[\tau_h]}\\
    %=&\lim_{h \searrow 0} \frac{\Ex_x\left[\int_0^{\tau_h} \left(\mathcal{A} -r\right)e^{-rs} J_{\tau^{D,\Psi}}(X_s) ds \right] }{\Ex_x[\tau_h]}
%\end{align*}
\begin{align*}
    &J_{ \theta_{\tau_h}\circ \tau^{D, \Psi}+\tau_h}(x)%= \Ex_{x}[e^{-r\tau_h} J_{\tau^{D,\Psi}}(X_{\tau_h}) ] 
    %=& e^{-r0}J_{\tau^{D,\Psi}}(x) + \Ex_x\left[\int_0^{\tau_h} \left(\mathcal{A} +\frac{\partial}{\partial s}\right)e^{-rs} J_{\tau^{D,\Psi}}(X_s) ds \right] \\
    = J_{\tau^{D,\Psi}}(x)+ \Ex_x\left[\int_0^{\tau_h} e^{-rs} \left( \mathcal{A} -r\right) J_{\tau^{D,\Psi}}(X_s) ds \right].
\end{align*}
With this we are able to make the next estimate using $e^{-z}-1 \ge -z$.
\begin{align*}
    &\left\vert\frac{ J_{ \theta_{\tau_h}\circ\tau^{D, \Psi}  +\tau_h}(x) -J_{\tau^{D,\Psi}} (x) }{\Ex_x[\tau_h]} - ( \mathcal{A} -r) J_{\tau^{D,\Psi}}(x) \right\vert \\[0.2cm]
    %=& \left\vert\frac{ \Ex_x\left[\int_0^{\tau_h} e^{-rs} \left( \mathcal{A} -r\right) J_{\tau^{D,\Psi}}(X_s) ds \right] }{\Ex_x[\tau_h]} - ( \mathcal{A} -r) J_{\tau^{D,\Psi}}(x) \right\vert \\
    %=& \left\vert\frac{ \Ex_x\left[\int_0^{\tau_h} e^{-rs} \left( \mathcal{A} -r\right) J_{\tau^{D,\Psi}}(X_s)  - ( \mathcal{A} -r) J_{\tau^{D,\Psi}}(x)ds \right]}{\Ex_x[\tau_h]}  \right\vert \\
    \le&\, \frac{ \Ex_x\left[\int_0^{\tau_h}\left\vert e^{-rs} \left( \mathcal{A} -r\right) J_{\tau^{D,\Psi}}(X_s)  -e^{-rs} ( \mathcal{A} -r) J_{\tau^{D,\Psi}}(x)\right\vert ds \right]}{\Ex_x[\tau_h]}   \\[0.2cm]
    &\, +\frac{ \Ex_x\left[\int_0^{\tau_h} \left\vert e^{-rs} ( \mathcal{A} -r) J_{\tau^{D,\Psi}}(x)  - ( \mathcal{A} -r) J_{\tau^{D,\Psi}}(x)\right\vert ds \right]}{\Ex_x[\tau_h]}   \\[0.2cm]
    %\le& \frac{ \Ex_x\left[\int_0^{\tau_h} e^{-rs} \sup_{y\in[x-h,x+h]}\left\vert \left( \mathcal{A} -r\right) J_{\tau^{D,\Psi}}(y)  -  ( \mathcal{A} -r) J_{\tau^{D,\Psi}}(x)\right\vert ds \right]}{\Ex_x[\tau_h]}   \\
    %&+\frac{ \Ex_x\left[\int_0^{\tau_h} \vert e^{-rs}-1  \vert \left\vert ( \mathcal{A} -r) J_{\tau^{D,\Psi}}(x) \right\vert ds \right]}{\Ex_x[\tau_h]}   \\
    \allowdisplaybreaks
    \le&\, \frac{ \Ex_x\left[\int_0^{\tau_h}  \sup_{y\in[x-h,x+h]}\left\vert \left( \mathcal{A} -r\right) J_{\tau^{D,\Psi}}(y)  -  ( \mathcal{A} -r) J_{\tau^{D,\Psi}}(x)\right\vert ds \right]}{\Ex_x[\tau_h]}   \\[0.2cm]
    &\, +\frac{ \Ex_x\left[\int_0^{\tau_h} rs  \left\vert ( \mathcal{A} -r) J_{\tau^{D,\Psi}}(x) \right\vert ds \right]}{\Ex_x[\tau_h]}   \\[0.2cm]
    \le&\, \sup_{y\in[x-h,x+h]}\left\vert \left( \mathcal{A} -r\right) J_{\tau^{D,\Psi}}(y)  -  ( \mathcal{A} -r) J_{\tau^{D,\Psi}}(x)\right\vert + r  \left\vert ( \mathcal{A} -r) J_{\tau^{D,\Psi}}(x) \right\vert \frac{\Ex_x[\tau_h^2]}{\Ex_x[\tau_h]} \to 0
\end{align*}
for $h \searrow 0$ by the $\ccc^2$ assumption on $J_{\tau^{D,\Psi}}$ together with $\frac{\Ex_x[\tau_h^2]}{\Ex_x[\tau_h]}\to 0$ for $h\searrow 0$ by \cite[Lemma 25]{BodnariuChristensenLindensjoe2022}. This means
\begin{align*}
     \lim_{h \searrow 0} \frac{ J_{ \theta_{\tau_h}\circ \tau^{D, \Psi} +\tau_h}(x) -J_{\tau^{D,\Psi}} (x) }{\Ex_x[\tau_h]}&= ( \mathcal{A} -r) J_{\tau^{D,\Psi}}(x).
\end{align*} Recalling \eqref{nun} this proves the claim. 
\end{beweis}

\begin{theorem} \label{veri}
Let $\tau^{D, \Psi} \in \mathcal{S}$ be a randomized Markovian time with regular $D$ and $\psi \in \mathfrak{V}_D$ the function defining $\Psi$ via \eqref{defpsi}. Suppose that the function $g$ is regular with respect to $D$. We also assume that $J_{\tau^{D,\Psi}}\vert_D$ is continuous.
If $\tau^{D, \Psi}$ fulfills the following conditions it is an equilibrium randomized Markovian time:
\begin{enumerate}[(i)]
    \item \label{i} $(\mathcal{A}-r)g(x) \le 0$ for all $x\in (D^c)^\circ$. 
    \item \label{ii} $\edp(x)=T$ for all $x \in (\spt(\psi))^\circ$.
    \item \label{iv} For each $x\in \partial D$ there exists some $h >0$ such that $ J_{\tau^{D,\Psi}}\vert_{(x-h,x]} \in \ccc^2((x-h,x])$ as well as $ J_{\tau^{D,\Psi}}\vert_{[x,x+h)} \in \ccc^2([x,x+h))$. Also $\partial_x J_{\tau^{D,\Psi}}(x+) \le \partial_x J_{\tau^{D,\Psi}}(x-)$ holds.
    \item \label{iii} $g(x) \le J_{\tau^{D,\Psi}}(x)$ and $\edp(x) \le T$ for all $x \in D$.
\end{enumerate}

\end{theorem}
\begin{remark}
\begin{itemize}
    \item The Verification Theorem \ref{veri} acts as a counterpart to Proposition \ref{propeq} to characterize all equilibria with regular $D$, continuous $J_{\tau^{D,\Psi}}\vert_D$ such that for all $x\in \partial D$ there is an $h>0$ with $J_{\tau^{D,\Psi}}\vert_{(x-h,x]}\in \ccc^2((x-h,x])$ and $J_{\tau^{D,\Psi}}\vert_{[x,x+h)}\in \ccc^2([x,x+h))$ as well as right continuous $g$ regular with respect to $D$ up to the case that $g(x)=J_{\tau^{D,\Psi}}(x)$ somewhere in $(\spt(\psi))^\circ$.

    \item $g\le J_{\tau^{D,\psi}}$ and $\edp\le T$ actually hold globally on $(\alpha,\beta)$, because $g(x) \le J_{\tau^{D,\Psi}}(x)$ is trivially fulfilled for $x \in D^c$ and $\edp\vert_{D^c} \equiv 0$.
    \item A sufficient condition for the continuity of $J_{\tau^{D,\Psi}}\vert_D$ is derived in Lemma \ref{Jstetigkeit}.
    %\item Condition \eqref{iv} is also basically a standard condition for optimal stopping times to ensure superharmonicity of the reward function at the boundary of the stopping set. 
\end{itemize}

\end{remark}

\begin{proof}[Proof of the Verification Theorem \ref{veri}]
We have to show \eqref{equi} for all $x\in(\alpha,\beta)$. Let $x\in (\alpha,\beta)$, $\tau \in \mathcal{S}$, $V\subset (\alpha,\beta)$ a neighborhood of $x$ and $r>0$ such that $\Ex_z[\tau^{D, \Psi} \diamond \tau(r')]\le T$ for all $z\in V$ and all $r'\in[0,r)$. 

$\boxed{\text{Case 1:}}$ For $x \in (D^c)^\circ$ we can proceed in almost exactly the same way as in the proof of Proposition \ref{propeq} part \eqref{a}. Instead of starting the calculation \eqref{glre} with the premise $0 \le \liminf_{h \searrow 0} \frac{J_{\tau^{D,\Psi}}(x) - J_{\tau^{D,\Psi} \diamond \tau(h)}(x)}{\Ex_x[\tau_h]}$ we have the assumption $(r-\mathcal{A})g(x) \ge 0$ to finish the calculation with and thus infer the claim. 
    
$\boxed{\text{Case 2:}}$ Consider $x\in D$. Assume $\tau = \tau^{\Tilde{D},\Tilde{\Psi}}$ for some open $\Tilde{D}\subset(\alpha,\beta)$ and $\Tilde{\Psi} \in \mathfrak{W}_{\tilde{D}}$ defined via \eqref{defpsi} with a function $\tilde{\psi}\in \mathfrak{V}_{\tilde{D}}$. If $x \not\in \Tilde{D}$, we have $J_{\tau^{D,\Psi} \diamond \tau^{\Tilde{D},\Tilde{\Psi}}(h)}(x) = g(x),$ $\PR_x$-a.s. Thus we can conclude \eqref{equi} from \eqref{ii}. Now we are left with the case $x \in \Tilde{D}$.
    %RCLL function $\Tilde{\psi}$ that has finitely many discontinuities and vanishes on $(\alpha,\beta)\setminus D$. 
    First we show 
    \begin{align}
        \Tilde{\psi}(x) \ge \psi(x) \label{aux3}
    \end{align}
    for all $x\in D $. If $\psi(x)=0$ this is clear because $\tilde{\psi}$ is non-negative by assumption. If $\psi(x)>0$ we argue by contradiction, so assume $\Tilde{\psi}(x) < \psi(x)$. By right continuity and the fact that $\psi, \tilde{\psi}$ have finitely many jumps there is some $y\in (\spt(\psi))^\circ \cap \Tilde{D} \cap V$ and $h\in(0,r)$ such that $\psi, \tilde{\psi}$ are continuous on $[y-h,y+h]\subset (\spt(\psi))^\circ \cap \Tilde{D}$ and $ \tilde{\psi}< \psi$ on $[y-h,y+h]$. 
    For $U$ defined as in Remark \ref{rem1} \eqref{rem15} this implies 
    \begin{align}
        \PR_y(\tau^{\tilde{D},\tilde{\Psi}} \le \tau_h < \tau^{D,\Psi})=\PR_y\bigg(\int_0^{\tau_h}\psi(X_s)ds< U \le \int_0^{\tau_h}\Tilde{\psi}(X_s)ds  \bigg) =0.\label{a1}
    \end{align} Similarly, additionally using $\tau_h \le \tau^D \wedge \tau^{\tilde{D}},$ $\PR_y$-a.s., we derive
    \begin{align}
        \tau^{D,\Psi} \wedge \tau_h %&= \inf \left\{ t \ge 0 : \int_0^t\psi(X_s) ds \ge U \right\} \wedge \tau_h \notag\\
        &= \inf \left\{ t \ge 0 \colon  \int_0^{t\wedge\tau_h}\psi(X_s) ds \ge U \right\} \wedge \tau_h \notag\\
        &\le \inf \left\{ t \ge 0 \colon  \int_0^{t\wedge\tau_h}\Tilde{\psi}(X_s) ds \ge U \right\} \wedge \tau_h %\notag\\
        %&= \inf \left\{ t \ge 0 : \int_0^{t}\Tilde{\psi}(X_s) ds \ge U \right\} \wedge \tau_h 
        = \tau^{\tilde{D}, \Tilde{\Psi}} \wedge \tau_h, \quad \PR_y\text{-a.s.} \label{zus1}
    \end{align}
    Next we prove $\Ex_y[\tau^{D,\Psi} \diamond \tau^{\Tilde{D},\Tilde{\Psi}} (h)]>\Ex_y[\tau^{D,\Psi}]=T $ which contradicts the assumption on $\tau^{\Tilde{D},\Tilde{\Psi}}= \tau$ we made in the beginning of this proof. Applying \eqref{a1} in the penultimate and \eqref{zus1} in the last step we obtain
    \begin{align*}
        &\Ex_y[\tau^{D,\Psi} \diamond \tau^{\Tilde{D},\Tilde{\Psi}} (h)]= \Ex_y[\ind{\tau^{\Tilde{D},\Tilde{\Psi}} \le \tau_h}\tau^{\Tilde{D},\Tilde{\Psi}} + \ind{\tau^{\Tilde{D},\Tilde{\Psi}}>\tau_h}(\theta_{\tau_h} \circ\tau^{D,\Psi}  + \tau_h)]\\[0.1cm]
        %=& \Ex_y[(\ind{\tau^{D,\Psi}\le \tau_h}+ \ind{\tau^{D,\Psi} > \tau_h})(\ind{\tau^{\Tilde{D},\Tilde{\Psi}} \le \tau_h}\tau^{\Tilde{D},\Tilde{\Psi}} + \ind{\tau^{\Tilde{D},\Tilde{\Psi}}>\tau_h}( \theta_{\tau_h}\circ \tau^{D,\Psi} + \tau_h))]\\
        %=& \Ex_y[\ind{\tau^{D,\Psi}\le \tau_h} \ind{\tau^{\Tilde{D},\Tilde{\Psi}} \le \tau_h}\tau^{D,\Psi} ] + \Ex_y[\ind{\tau^{D,\Psi}\le \tau_h} \ind{\tau^{\Tilde{D},\Tilde{\Psi}} \le \tau_h}(\tau^{\Tilde{D},\Tilde{\Psi}}-\tau^{D,\Psi}) ] \\
        %+& \Ex_y[\ind{\tau^{D,\Psi}\le \tau_h}\ind{\tau^{\Tilde{D},\Tilde{\Psi}}>\tau_h}\tau^{D,\Psi}]+\Ex_y[\ind{\tau^{D,\Psi}\le \tau_h}\ind{\tau^{\Tilde{D},\Tilde{\Psi}}>\tau_h}( \theta_{\tau_h}\circ \tau^{D,\Psi} + \tau_h-\tau^{D,\Psi})] \\
        %+& \Ex_y[\ind{\tau^{D,\Psi} > \tau_h}\ind{\tau^{\Tilde{D},\Tilde{\Psi}} \le \tau_h}( \theta_{\tau_h}\circ \tau^{D,\Psi} + \tau_h)] + \Ex_y[\ind{\tau^{D,\Psi} > \tau_h}\ind{\tau^{\Tilde{D},\Tilde{\Psi}} \le \tau_h} \\
        %\cdot&(\tau^{\Tilde{D},\Tilde{\Psi}}- \theta_{\tau_h}\circ \tau^{D,\Psi} - \tau_h)]+ \Ex_y[\ind{\tau^{D,\Psi} > \tau_h}\ind{\tau^{\Tilde{D},\Tilde{\Psi}}>\tau_h}(\theta_{\tau_h} \circ \tau^{D,\Psi} + \tau_h)]\\
        =& \,\Ex_y[\tau^{D,\Psi}] + \Ex_y[\ind{\tau^{D,\Psi}\le \tau_h} \ind{\tau^{\Tilde{D},\Tilde{\Psi}} \le \tau_h}(\tau^{\Tilde{D},\Tilde{\Psi}}-\tau^{D,\Psi}) ] \\[0.1cm]
        &\,+ \Ex_y[\ind{\tau^{D,\Psi}\le \tau_h}\ind{\tau^{\Tilde{D},\Tilde{\Psi}}>\tau_h}(\theta_{\tau_h} \circ \tau^{D,\Psi} + \tau_h-\tau^{D,\Psi})] \\[0.1cm]
        &\, +\Ex_y[\ind{\tau^{D,\Psi} > \tau_h}\ind{\tau^{\Tilde{D},\Tilde{\Psi}} \le \tau_h}  (\tau^{\Tilde{D},\Tilde{\Psi}}-  \tau^{D,\Psi} )]\\[0.1cm]
        %=& T + \Ex_y[\ind{\tau^{D,\Psi}\le \tau_h} \ind{\tau^{\Tilde{D},\Tilde{\Psi}} \le \tau_h}(\tau^{\Tilde{D},\Tilde{\Psi}}-\tau^{D,\Psi}) ] \\
        %+& \Ex_y[\ind{\tau^{D,\Psi}\le \tau_h}\ind{\tau^{\Tilde{D},\Tilde{\Psi}}>\tau_h} \theta_{\tau_h} \circ\tau^{D,\Psi} ] - \Ex_y[\ind{\tau^{D,\Psi} > \tau_h}\ind{\tau^{\Tilde{D},\Tilde{\Psi}} \le \tau_h}   \theta_{\tau_h}\circ \tau^{D,\Psi} ] \\
        %-& \Ex_y[\ind{\tau^{D,\Psi}\le \tau_h}\ind{\tau^{\Tilde{D},\Tilde{\Psi}}>\tau_h}\tau^{D,\Psi}]+ \Ex_y[\ind{\tau^{D,\Psi} > \tau_h}\ind{\tau^{\Tilde{D},\Tilde{\Psi}} \le \tau_h}  \tau^{\Tilde{D},\Tilde{\Psi}}] \\
        %+& \Ex_y[\ind{\tau^{D,\Psi}\le \tau_h}\ind{\tau^{\Tilde{D},\Tilde{\Psi}}>\tau_h} \tau_h]- \Ex_y[\ind{\tau^{D,\Psi} > \tau_h}\ind{\tau^{\Tilde{D},\Tilde{\Psi}} \le \tau_h}  \tau_h]\\
        %Hier lassen wir nichtnegative Terme weg und verwenden $\PR_y(\tau^{\tilde{D},\tilde{\Psi}} \le \tau_h < \tau^{D,\Psi})=0$
        %\stackrel{*}{\ge}
        \ge&\, T + \Ex_y[\ind{\tau^{D,\Psi}\le \tau_h} \ind{\tau^{\Tilde{D},\Tilde{\Psi}} \le \tau_h}(\tau^{\Tilde{D},\Tilde{\Psi}}-\tau^{D,\Psi}) ] 
        + \Ex_y[\ind{\tau^{D,\Psi}\le \tau_h}\ind{\tau^{\Tilde{D},\Tilde{\Psi}}>\tau_h} (\tau_h-\tau^{D,\Psi})] \\
        >&\, T.
    \end{align*} 
    %At the point marked by $*$ we drop two non-negative terms and use \eqref{a1} to get rid of two non-positive terms that include the indicator $\ind{\tau^{D,\Psi} > \tau_h}\ind{\tau^{\Tilde{D},\Tilde{\Psi}} \le \tau_h}$. 
    This shows \eqref{aux3}. Now we verify the equilibrium condition \eqref{equi}. If $x\in \tilde{D}$ we apply Lemma \ref{Kri10} and Remark \ref{t5} \eqref{t51} followed by \eqref{aux3} together with condition \eqref{iii} from this theorem and obtain
    \begin{align*}
        &\liminf_{h \searrow 0} \frac{J_{\tau^{D,\Psi}}(x) - J_{\tau^{D,\Psi} \diamond \tau^{\tilde{D},\tilde{\Psi}}(h)}(x)}{\Ex_x[\tau_h]} \\
        = & \,\liminf_{h \searrow 0} \frac{J_{\tau^{D,\Psi}}(x) - J_{\theta_{\tau_h}\circ\tau^{D,\Psi} + \tau_h }(x)}{\Ex_x[\tau_h]} + \frac{J_{\theta_{\tau_h}\circ \tau^{D,\Psi}+ \tau_h }(x) - J_{\tau^{D,\Psi} \diamond \tau^{\tilde{D},\tilde{\Psi}}(h)}(x)}{\Ex_x[\tau_h]}\\
        %=& -\frac{1}{2}\psi(x-)(J_{\tau^{D,\Psi}}(x)-g(x-) ) -\frac{1}{2} \psi(x)(J_{\tau^{D,\Psi}}(x)- g(x+) )   \\ +&\frac{1}{2}\tilde{\psi}(x-)(J_{\tau^{D,\Psi}}(x)-g(x-) ) +\frac{1}{2} \tilde{\psi}(x)(J_{\tau^{D,\Psi}}(x)- g(x+) )\\
        =& \,\frac{1}{2}(( \tilde{\psi}(x-)- \psi(x-) )(J_{\tau^{D,\Psi}}(x-)-g(x-)) + ( \tilde{\psi}(x) -\psi(x) )(J_{\tau^{D,\Psi}}(x+)-g(x+)))\\
        %=& -\frac{1}{2}(\psi(x-)+ \psi(x) )(J_{\tau^{D, \Psi}}(x)- g(x))+ \frac{1}{2}(\tilde{\psi}(x-)+ \tilde{\psi}(x) )(J_{\tau^{D, \Psi}}(x)- g(x))\\
        %=& \frac{1}{2} (J_{\tau^{D, \Psi}}(x)- g(x)) (\tilde{\psi}(x-)+ \tilde{\psi}(x) - (\psi(x-)+ \psi(x)))
        \ge&\, 0.
    \end{align*}
    If $x \not \in \tilde{D}$ we conclude that
    \begin{align*}
        &\liminf_{h \searrow 0} \frac{J_{\tau^{D,\Psi}}(x) - J_{\tau^{D,\Psi} \diamond \tau^{\tilde{D},\tilde{\Psi}}(h)}(x)}{\Ex_x[\tau_h]} =  \liminf_{h \searrow 0} \frac{J_{\tau^{D,\Psi}}(x) - g(x)}{\Ex_x[\tau_h]} \ge 0
    \end{align*} due to condition \eqref{iii}.
    
    $\boxed{\text{Case 3:}}$ Consider $x \in \partial D$. Once again let $\tau = \tau^{\Tilde{D},\Tilde{\Psi}}$ for some open $\Tilde{D}\subset(\alpha,\beta)$ and $\Tilde{\Psi} \in \mathfrak{W}_D$. Clearly $e_{D,\Psi}(x) = \Ex_x[\tau^{D,\Psi}]=0$. Thus by \eqref{ii} and the continuity of $e_{D,\Psi}$ (recall Lemma \ref{rightcont}) the support of $\psi$  has no interior points near $x$. By right continuity of $\psi$ there is some $\delta>0$ such that $\psi \vert_{(x-\delta,x+\delta)}\equiv 0$.

    Using \eqref{aux2} in the first step, condition \eqref{iii} in the second one followed by a variation of Itô's formula (cf.\ \cite{Peskir2007}) in the third as well as condition \eqref{iv}  and Lemma \ref{Kri16} (this can be applied for $h$ so small that $J_{\tau^{D,\Psi}}\vert_{[x,x+h)},J_{\tau^{D,\Psi}}\vert_{(x-h,x]}\in \ccc^2$) in the fourth we obtain
    \begin{align}
        \liminf_{h \searrow 0}& \frac{J_{\tau^{D,\Psi}}(x) - J_{\tau^{D,\Psi} \diamond \tau^{\tilde{D},\tilde{\Psi}}(h)}(x)}{\Ex_x[\tau_h]} \notag\\
        = \liminf_{h \searrow 0}& \frac{g(x) - \Ex_{x}\big[ \ind{\tau^{\tilde{D},\tilde{\Psi}} \le \tau_h } e^{-r \tau^{\tilde{D},\tilde{\Psi}}} g(X_{\tau^{\tilde{D},\tilde{\Psi}}}) + \ind{\tau^{\tilde{D},\tilde{\Psi}} > \tau_h}e^{-r \tau_h} J_{\tau^{D,\Psi}}(X_{\tau_h})\big] }{\Ex_x[\tau_h]}\notag\\
        %\ge \liminf_{h \searrow 0}& \frac{g(x) - \Ex_{x}[ \ind{\tau^{\tilde{D},\tilde{\Psi}} \le \tau_h } e^{-r \tau^{\tilde{D},\tilde{\Psi}}} J_{\tau^{D,\Psi}}(X_{\tau^{\tilde{D},\tilde{\Psi}}}) + \ind{\tau^{\tilde{D},\tilde{\Psi}} > \tau_h}e^{-r \tau_h} J_{\tau^{D,\Psi}}(X_{\tau_h})] }{\Ex_x[\tau_h]}\notag\\
        \ge \liminf_{h \searrow 0}& \frac{g(x) - \Ex_{x}[  e^{-r \tau^{\tilde{D},\tilde{\Psi}}\wedge \tau_h} J_{\tau^{D,\Psi}}(X_{\tau^{\tilde{D},\tilde{\Psi}}\wedge \tau_h})  ]}{\Ex_x[\tau_h]}\notag\\
        =  \liminf_{h \searrow 0}& \frac{g(x) - \Ex_{x}\left[ J_{\tau^{D,\Psi}}(x) +  \int_0^{\tau^{\tilde{D},\tilde{\Psi}}\wedge \tau_h} (\mathcal{A}-r)  e^{-rs} J_{\tau^{D,\Psi}}(X_{s}) ds \right]}{\Ex_x[\tau_h]}\notag\\
        -&\frac{\frac{1}{2}\Ex_x \left[\int_0^{\tau^{\tilde{D},\tilde{\Psi}}\wedge \tau_h} \partial_x J_{\tau^{D,\Psi}}(x+)-  \partial_x J_{\tau^{D,\Psi}}(x-) dL_s^x\right]}{\Ex_x[\tau_h]}\notag\\
        %=  \liminf_{h \searrow 0} &\frac{\Ex_{x}\left[ \int_0^{\tau^{\tilde{D},\tilde{\Psi}}\wedge \tau_h} \ind{X_s\neq x}(r-\mathcal{A})  e^{-rs} J_{\tau^{D,\Psi}}(X_{s}) ds\right]}{\Ex_x[\tau_h]}\notag\\
        %-&\frac{\frac{1}{2}\Ex_x \left[\int_0^{\tau^{\tilde{D},\tilde{\Psi}}\wedge \tau_h} \partial_x J_{\tau^{D,\Psi}}(x+)-  \partial_x J_{\tau^{D,\Psi}}(x-) dL_s^x\right]}{\Ex_x[\tau_h]}\notag\\
        %\ge \liminf_{h \searrow 0} & \frac{\Ex_{x}\left[ \int_0^{\tau^{\tilde{D},\tilde{\Psi}}\wedge \tau_h} \ind{X_s\neq x}(r-\mathcal{A})  e^{-rs} J_{\tau^{D,\Psi}}(X_{s}) ds\right]}{\Ex_x[\tau_h]}\notag\\
        %= \liminf_{h \searrow 0}& \frac{\Ex_{x}\left[ \int_0^{\tau^{\tilde{D},\tilde{\Psi}}\wedge \tau_h} \ind{X_s \in (D^c)^\circ} e^{-rs}(r-\mathcal{A})  J_{\tau^{D,\Psi}}(X_{s}) ds\right]}{\Ex_x[\tau_h]}\notag\\
        %+&\frac{\Ex_{x}\left[ \int_0^{\tau^{\tilde{D},\tilde{\Psi}}\wedge \tau_h} \ind{X_s \in D} e^{-rs}(r-\mathcal{A})   J_{\tau^{D,\Psi}}(X_{s}) ds\right]}{\Ex_x[\tau_h]}\notag\\
        % = \liminf_{h \searrow 0}& \frac{\Ex_{x}\left[ \int_0^{\tau^{\tilde{D},\tilde{\Psi}}\wedge \tau_h} \ind{X_s \in (D^c)^\circ} e^{-rs}(r-\mathcal{A})  g(X_{s}) ds\right]}{\Ex_x[\tau_h]}\notag\\
        %+&\frac{\Ex_{x}\left[ \int_0^{\tau^{\tilde{D},\tilde{\Psi}}\wedge \tau_h} \ind{X_s \in D} e^{-rs}(r-\mathcal{A})   J_{\tau^{D,\Psi}}(X_{s}) ds\right]}{\Ex_x[\tau_h]}\notag\\
        %\ge \liminf_{h \searrow 0}& \frac{\Ex_{x}\left[ \int_0^{\tau^{\tilde{D},\tilde{\Psi}}\wedge \tau_h} \ind{X_s \in D} e^{-rs}(r-\mathcal{A})   J_{\tau^{D,\Psi}}(X_{s}) ds\right]}{\Ex_x[\tau_h]}\notag\\
        \ge \liminf_{h \searrow 0}& \frac{\Ex_{x}\left[ \int_0^{\tau^{\tilde{D},\tilde{\Psi}}\wedge \tau_h} \ind{X_s \in (D^c)^\circ} e^{-rs}(r-\mathcal{A})  g(X_{s}) ds\right]}{\Ex_x[\tau_h]}\notag\\
        +\liminf_{h \searrow 0}& \frac{\Ex_{x}\left[ \int_0^{\tau^{\tilde{D},\tilde{\Psi}}\wedge \tau_h} -\ind{X_s \in D} e^{-rs} \psi(X_s)(J_{\tau^{D,\Psi}}(X_s)-\frac{1}{2}(g(X_s-)+g(X_s+)))  ds\right]}{\Ex_x[\tau_h]}. \notag
    \end{align}
    The first summand in the last line is non-negative by \eqref{i}. %Above we have shown the existence of some 
    Recall that there exists $\delta >0$ such that  $\psi \vert_{(x-\delta,x+\delta)}\equiv 0$. This makes the integrand of the second summand in the last line vanish for $h$ close to 0. Together this gives the claim.
\end{proof}

\begin{remark}
From the first part of Case 3 in the proof of the Verification Theorem \ref{veri} we read off that any randomized Markovian time in $\mathcal{S}(T)$ that satisfies condition \eqref{ii} of Theorem \ref{veri} has zero rate of randomized stopping in a neighborhood of the stopping region because there is still enough time to continue until we actually reach the stopping region. 
\end{remark}

\section{Finding equilibria} \label{Ch4}
%We first derive further properties of (equilibrium) randomized Markovian times to provide more information for a guess-and-verify approach to finding equilibria.
We derive further properties of (equilibrium) randomized Markovian times and  provide a guess-and-verify approach for finding equilibria.

\begin{proposition}\label{fixpsi}
    Let $\tau^{D, \Psi} \in \mathcal{S}(T)$ be a randomized Markovian time and $\psi \in \mathfrak{V}_D$ the function defining $\Psi$ via \eqref{defpsi}. 
    \begin{enumerate}[(i)]
    \item \label{fixpsi1} For all open $U\subset D$ such that $\psi\vert_U \equiv 0$ we have $\edp\vert_U \in \ccc^2(U)$ and
    \begin{align*}
        \mathcal{A}\edp (x) = -1, \qquad x\in U.
    \end{align*}
    \item \label{fixpsi2} Suppose condition \eqref{ii} from the Verification Theorem \ref{veri} holds for $\tau^{D,\Psi}$, i.e.\ $\edp(x)=T$ for all $x \in (\spt(\psi))^\circ$.
    \begin{enumerate}[a)]
        \item \label{fixpsi2a} For all $x \in (\alpha, \beta)$ we have
            \begin{align}
                \psi(x) \in \left\{0, \frac{1}{T}\right\}.\notag
            \end{align}
        \item\label{above}  Let $U_1,U_2,...\subset (\alpha, \beta)$ be the pairwise disjoint open, connected components of $D$, i.e.\ $D= \bigcupdot_{n\in \N} U_n$ (note that $U_n = \varnothing$ is allowed). For each $n \in \N$ we either have $U_n\cap\{\psi= \frac{1}{T}\} = \varnothing$ or 
        \begin{enumerate}[1.]
            \item $U_n\cap\{\psi= \frac{1}{T}\} = (\alpha,\beta)$ if $\inf U_n =\alpha$, $\sup U_n = \beta$,
            \item $U_n\cap\{\psi= \frac{1}{T}\} = (\alpha, b_n)$ for some $b_n\in(\alpha,\beta)$ if $\inf U_n = \alpha$, $\sup U_n \in(\alpha, \beta)$,
            \item $U_n\cap\{\psi= \frac{1}{T}\} = [a_n,\beta)$ for some $a_n\in(\alpha,\beta)$ if $\inf U_n \in(\alpha,\beta)$, $\sup U_n = \beta$,
            \item $U_n\cap\{\psi= \frac{1}{T}\} = [a_n,b_n)$ for some $a_n<b_n\in(\alpha,\beta)$ if $\inf U_n,\sup U_n \in(\alpha, \beta)$.
        \end{enumerate}
        % there holds one of the following:
        % \begin{enumerate}[1.]
        %     \item $U_n\cap\{x\in(\alpha,\beta):\psi(x)= \frac{1}{T}\} = \varnothing$,
        %     \item $U_n\cap\{x\in(\alpha,\beta):\psi(x)= \frac{1}{T}\} = (\alpha,\beta)$
        %     \item $U_n\cap\{x\in(\alpha,\beta):\psi(x)= \frac{1}{T}\} = (\alpha, b_n)$ for some $b_n\in(\alpha,\beta)$,
        %     \item $U_n\cap\{x\in(\alpha,\beta):\psi(x)= \frac{1}{T}\} = [a_n,\beta)$ for some $a_n\in(\alpha,\beta)$,
        %     \item $U_n\cap\{x\in(\alpha,\beta):\psi(x)= \frac{1}{T}\} = [a_n,b_n)$ for some $a_n,b_n\in(\alpha,\beta)$.
        % \end{enumerate} 
        \item \label{fixpsi4} $\edp'(a_n-)=0$, $\edp'(b_n+)=0$ in the cases 2, 3 and 4 from \eqref{above}.
        %\item \label{fixpsi3}In case $U_n\cap\{x\in(\alpha,\beta):\psi(x)= \frac{1}{T}\} \neq \varnothing$ the constants $a_n$ and $b_n$ from \eqref{above} are uniquely determined by the equations
        %\begin{align*}
        %    &\edp(a_n)=T, &&\edp'(a_n+)=0,\\
        %    &\edp(b_n)=T, &&\edp'(b_n-)=0.
        %\end{align*}
    \end{enumerate}
\end{enumerate}
\end{proposition}

\begin{beweis}
    By \cite[Theorem 13.16]{dynkin_markov2} the function
    \begin{align*}
        f_\varphi(x):= \Ex_x\left[\int_0^{\tau^U} e^{-\int_0^t \psi(X_s)ds} dt\right]+ \Ex_x\left[e^{-\int_0^{\tau^U} \psi(X_s) ds} \varphi(X_{\tau^U})\right]
    \end{align*} 
    is the unique $\ccc^2$-solution to $\mathcal{A}f-\psi f=-1$ on $U$ that satisfies the boundary condition $\lim_{y\to z}f(y) = \varphi(z)$ for all $z\in \partial U \subset (\alpha,\beta)$ whenever $\psi\vert_U$ is Hölder-continuous. Moreover, by \cite[Theorem 13.11]{dynkin_markov2} and \cite[Theorem 13.12]{dynkin_markov2}, if $\psi\vert_U$ is merely continuous, the function $f_\varphi$ is a continuous solution to $\overline{\mathcal{A}}f-\psi f =-1$ on $U$, with $\overline{\mathcal{A}}$ denoting the characteristic operator in the sense of \cite[Chapter 5, \S 3]{dynkin_markov}.\\
    
    By the Markov property, Remark \ref{rem1} \eqref{conddistr} and the fact that $\Ex_{x,n}[\tau^{D, \Psi}] = 0$ for all $n\in \N$ we get
    \begin{align*}
        \edp(x) &= \Ex_x[\tau^{D,\Psi}] = \Ex_x[\tau^U\wedge \tau^{D,\Psi}]+ \Ex_x[\theta_{ \tau^U \wedge \tau^{D, \Psi}} \circ \tau^{D, \Psi}]\\[0.1cm]
        &= \Ex_x\left[ \Ex_x\left[ \int^{\tau^U}_0 \ind{t < \tau^\Psi} dt \bigg\vert \f_\infty^X \right]\right] + \Ex_x\left[ \Ex_{x,0}\left[ \theta_{ \tau^U \wedge \tau^{D, \Psi}} \circ \tau^{D, \Psi} \vert \f_{ \tau^U \wedge \tau^{D, \Psi}} \right]  \right]\\[0.1cm]
        &=\Ex_x\left[\int^{\tau^U}_0 \Ex_x\left[  \ind{t < \tau^\Psi}  \big\vert \f_\infty^X \right]dt\right] + \Ex_x\left[ \Ex_{X_{\tau^U \wedge \tau^{D, \Psi}},N_{\Psi_{\tau^U \wedge \tau^{D, \Psi}}}}[ \tau^{D, \Psi}] \right] \\[0.1cm]
        &=\Ex_x\left[ \int_0^{\tau^U} e^{-\Psi_t} dt  \right] + \Ex_x \left[ \ind{\tau^U < \tau^\Psi} \Ex_{X_{\tau^U}}[\tau^{D, \Psi}] \right] \\[0.1cm]
        &= \Ex_x\left[ \int_0^{\tau^U} e^{-\int_0^t\psi(X_s)ds} dt \right]+ \Ex_x\left[e^{-\int_0^{\tau^U}\psi(X_s)ds}\edp(X_{\tau^U})\right] =  f_{\edp}(x).
    \end{align*}
    Together with the fact that  $e_{D,\Psi}$ is continuous by Lemma \ref{rightcont} \eqref{rightcont_1} this already implies \eqref{fixpsi1}.
     Assume that there is an open neighborhood $U\subset (\spt(\psi))^\circ$ of $x$ such that $\psi\vert_U$ is continuous. Since $ f_{\edp}\vert_U=\edp\vert_U\equiv T$ on $U$ we have
    \begin{align}\label{eq:on U}
        0= \overline{\mathcal{A}}f_{\edp} = \psi f_{\edp} - 1=\psi T-1 \qquad \text{on $U$.} 
    \end{align}
    Now let $x\in(\spt(\psi))^\circ$ be arbitrary. Since $\psi\in \mathfrak{V}_D$ has only finitely may discontinuities there is some open $U\subset (\spt(\psi))^\circ$ such that $x= \inf U$ and $\psi$ is continuous on $U$. Now $\psi\vert_U\equiv \frac{1}{T}$ by \eqref{eq:on U}. Right continuity of $\psi$ yields $\psi(x)=\frac{1}{T}$. Lastly the boundary points of $\spt(\psi)$ are either $0$ or $\frac{1}{T}$ by right continuity of $\psi$, which shows \eqref{fixpsi2a}.
    
For the proof of \eqref{above} first note that $U_n\cap\{\psi= \frac{1}{T}\}$ is connected. Indeed, if not, by \eqref{fixpsi2a} and the right continuity of $\psi$ there is a non-empty interval $[x,y)$ and some $\delta>0$ such that $(x-\delta,y+\delta)\subset U_n$, $\psi\vert_{[x,y)} \equiv 0$ and $\psi\vert_{(x-\delta,x)} \equiv \frac{1}{T}$ or $\psi\vert_{[y,y+\delta)} \equiv \frac{1}{T}$. The assumption of condition \eqref{ii} from the Verification Theorem and continuity of $\edp$ (cf.\ Lemma \ref{rightcont}) imply $\edp(x)=T=\edp(y)$. With that we reach a contradiction to $\tau^{D, \Psi} \in \mathcal{S}(T)$ via
\begin{align*}
    \edp\left(\frac{x+y}{2}\right) = \Ex_{\frac{x+y}{2}}[\tau^{(x,y)}] + \Ex_{X_{\tau^{(x,y)}}}[\tau^{D, \Psi}]=\Ex_{\frac{x+y}{2}}[\tau^{(x,y)}] +T > T.
\end{align*}
Hence, $U_n\cap\{\psi= \frac{1}{T}\}$ is connected. 
If $\alpha <\inf U_n<\sup U_n < \beta$ this suffices to show \eqref{above} by right continuity of $\psi$. On the other hand if $\inf U_n =\alpha$  and $\psi\vert_{(\alpha,\alpha+\delta)} \equiv 0 $ or $\sup U_n=\beta$ and $\psi\vert_{(\beta-\delta,\beta)} \equiv 0 $ for some $\delta>0$, respectively, then just as before we get $\edp(\alpha+\frac{\delta}{2})>T$ or $\edp(\beta - \frac{\delta}{2})>T$ respectively to reach a contradiction, which proves \eqref{above}.

Without loss of generality we show \eqref{fixpsi4} only for $a_n$. We will argue by contradiction, so set $L:= \inf U_n$, $a:=a_n$ and suppose $\edp'(a-)\neq 0$. We have already shown $\edp(a)=T$. If $\edp'(a-) < 0$, then $\edp>T$ on the left of $a$ which contradicts the assumption $\tau^{D, \Psi} \in \mathcal{S}(T)$. Thus we are left to treat the case $\edp'(a-)>0$. Let $\delta>0$ such that $a+\delta \in (\spt(\psi))^\circ$ and let $f\in \ccc^2([L,a+\delta])$ with $\mathcal{A}f \equiv 1$. Moreover, we denote the scale function of $X$ by $s$. With that, using Dynkin's formula for $y\in[L,a]$ we obtain
\begin{align*}
    \edp(y) &= \Ex_y[\tau^{(L,a)}] + \PR_y(X_{\tau^{(L,a)}}=a) T = \Ex_y[ f(X_{\tau^{(L,a)}})]-f(y) +\frac{s(y)-s(L)}{s(a)-s(L)}T\\
    &= \frac{s(a)-s(y)}{s(a)-s(L)} f(L) + \frac{s(y)-s(L)}{s(a)-s(L)}f(a)-f(y) + \frac{s(y)-s(L)}{s(a)-s(L)} T.
\end{align*} 
Note that the right-hand side of this equation is well defined and differentiable for $y\in[L,a+\delta]$ since $f$ and $s$ are. For $(z,y)\in[L,a+\delta]^2$ we define
\begin{align*}
    h(z,y):=&\,\frac{\partial}{\partial y} \left(  \frac{s(z)-s(y)}{s(z)-s(L)} f(L) + \frac{s(y)-s(L)}{s(z)-s(L)}f(z) -f(y) + \frac{s(y)-s(L)}{s(z)-s(L)} T \right) \\[0.15cm]
    =&\, -\frac{s'(y)f(L)}{s(z)-s(L)} + \frac{s'(y)f(z)}{s(z)-s(L)}- f'(y) + \frac{s'(y) T}{s(z)-s(L)}.
\end{align*}
Now we introduce an auxiliary function $H\colon [L,a+\delta] \to \R$ by
\begin{align*}
    y \mapsto \,&\Ex_y[\tau^{(L,a+\delta)}] + \PR_y (X_{\tau^{(L,a+\delta)}} = a + \delta ) T \\[0.15cm]
    &= \frac{s(a+\delta)-s(y)}{s(a+\delta)-s(L)} f(L) + \frac{s(y)-s(L)}{s(a+\delta)-s(L)}f(a+\delta)-f(y) + \frac{s(y)-s(L)}{s(a+\delta)-s(L)} T.
\end{align*} 
By definition
\begin{align}
    \edp(y) &= \Ex_y[\tau^{(L,a+\delta)} \wedge \tau^\Psi] + \Ex_y[ \ind{\tau^\Psi > \tau^{(L,a+\delta)}} \ind{X_{\tau^{(L,a+\delta)}}=a+\delta}] \Ex_{a+\delta}[\tau^{D,\Psi}] \notag\\
    &\le \Ex_y[\tau^{(L,a+\delta)}] + \PR_y(X_{\tau^{(L,a+\delta)}}=a+\delta) T = H(y) \label{hgl}
\end{align} for all $y\in[L,a+\delta]$. $H'$ is given by
\begin{align*}
    H'(y)= -\frac{s'(y)f(L)}{s(a+\delta)-s(L)} + \frac{s'(y)f(a+\delta)}{s(a+\delta)-s(L)}- f'(y) + \frac{s'(y) T}{s(a+\delta)-s(L)} = h(a+\delta,y).
\end{align*}  
Since $s'$ is continuous,  $h$ is continuous in $(a,a)$ with $h(a,a) = \edp'(a-)>0$. Thus for sufficiently small $\delta>0$ we find that $H'\vert_{[a,a+\delta]}>0$. Using \eqref{hgl} this yields a contradiction via $T=H(a+\delta) > H(a) \ge \edp(a) = T$, which proves \eqref{fixpsi4}.
\end{beweis}

\begin{remark}\label{scheme}
   Proposition \ref{fixpsi} gives us a method to construct an equilibrium $\tau^{D,\Psi}$ with connected $D$. By \eqref{above} a candidate for $\psi \not \equiv 0$ that satisfies \eqref{ii} from the Verification Theorem has the form $\frac{1}{T}\inda{[a,b)}$, $\frac{1}{T}\inda{(\alpha,b)}$, $\frac{1}{T}\inda{[a,\beta)}$ or $\frac{1}{T}\inda{(\alpha,\beta)}$ for some $\alpha<a<b<\beta$. By \eqref{fixpsi1}, \eqref{above} and \eqref{fixpsi4} we find that $a$ and $b$ depend on the boundaries of $D$ via
    \begin{align}
        &\mathcal{A}\edp=-1 \quad \text{on }\,(\inf D, a)\notag\\
        &\edp(\inf D)=0,\quad \edp(a)=T,\quad \edp'(a)=0 \label{sys1}
    \end{align} if $\inf D \in (\alpha,\beta)$ and
     \begin{align}
        &\mathcal{A}\edp=-1 \quad \text{on }\,(b, \sup D)\notag\\
        &\edp(\sup D)=0,\quad \edp(b)=T,\quad \edp'(b)=0\label{sys2}
    \end{align} if $\sup D \in (\alpha,\beta)$ respectively, while $a = \alpha$ if $\inf D = \alpha$ and $b = \beta$ if $\sup D = \beta$. Given $D$ basic ODE theory does not provide existence nor uniqueness of solutions to the boundary value problems \eqref{sys1} and \eqref{sys2}. Nevertheless, we have three equations and one free variable, namely $a$ or $b$, it seems reasonable to believe that the systems \eqref{sys1} and \eqref{sys2} have exactly one solution. In that case, similar to classical optimal stopping we rely on the smooth fit principle from Theorem \ref{veri} \eqref{iv} in order to determine the boundary of $D$ of the equilibrium randomized Markovian time.
\end{remark}

\begin{remark}
The scheme described in Remark \ref{scheme} can even be employed in situations where  $D$ is not an interval by applying it to the connected components of $D$, see Example \ref{ex:D_kein_Intervall}.
\end{remark}

\section{Examples} \label{Ch5}

The aim %of the rest
of this section is to give some examples for the scheme proposed in Remark \ref{scheme}. We start with the most simple case.
\begin{remark}\label{lambda_const}
%As Proposition \ref{fixpsi} rules out most of the candidate stopping rates $\psi$ for our Theorem \ref{veri}, in some situations a guess and verify approach already yields insights. \\
If $g$ continuous with at most polynomial growth and $(e^{-rt}g(X_t))_{t\in [0, \infty)}$ is a submartingale, then $\tau^*=\tau^{D,\Psi}$ with $D=(\alpha,\beta)$ and $\psi\equiv \frac{1}{T}$ defining $\Psi$ via \eqref{defpsi} is an equilibrium in the sense of Definition \ref{equil2}. \\
Here in view of  Remark \ref{rem1} \eqref{rem15} condition \eqref{ii} from the Verification Theorem is easily seen to hold via
\begin{align*}
    \Ex_x[\tau^*]= \Ex_x[\tau^{D,\Psi}]=\Ex_x[\tau^{\Psi}]= \Ex_x\left[\inf\left\{ t \ge 0 \colon  \frac{1}{T}t\ge U\right\}\right] = \Ex_x[T\cdot U] = T.
\end{align*} 
Since $(\alpha,\beta)=\spt(\psi)$ it only remains to show $g(x)\le J_{\tau^*}(x)$ from condition \eqref{iii} in Theorem \ref{veri}. As $(e^{-rt}g(X_t))_{t\in [0, \infty)}$ is a submartingale we apply the optional sampling theorem to obtain %Using $\tau^*=T\cdot U$ followed by Fubini's Theorem and the submartingale property of $(e^{-rt}g(X_t))_{t\in [0, \infty)}$ we get
\begin{align*}
    J_{\tau^*}(x) = \Ex_x[e^{-r\tau^*}g(X_{\tau^*})]\ge g(x).
\end{align*}
% \begin{align*}
%     J_{\tau^*}(x)&=\Ex_x[e^{-rt}g(X_{\tau^*})]=\Ex_x[e^{-rt}g(X_{T\cdot U})] = \Ex_x[\Ex_x[e^{-rt}g(X_{T\cdot U})\vert \f^X_\infty]]  \\
%     &= \Ex_x \left[ \int_0^\infty e^{-rt}g(X_{T\cdot u}) e^{-u} du \right]= \int_0^\infty \Ex_x[e^{-rt}g(X_{T\cdot u}) ]e^{-u} du\\
%     &\ge \int_0^\infty e^{-r0}g(x)e^{-u} du = g(x).
% \end{align*}
Therefore the Verification Theorem \ref{veri} yields the claim.
\end{remark}

The next example deals with a special case of Remark \ref{lambda_const} and shows that considering randomized Markovian times instead of only first exit times of sets that are regular in the sense of Definition \ref{regular} as admissible strategies leads to a truly more general concept of equilibrium.
\begin{ex}\label{nofe}
    Let $\alpha:= -\infty$, $\beta:= \infty$, $X=W=(W_t)_{t\in [0,\infty)}$ a standard Brownian motion, $g(x):=x^2$ and $r=0$. In the corresponding equilibrium problem there are no equilibria $\tau^D=\tau^{D,\Psi}\in \mathcal{S}(T)$ with $\Psi \equiv 0$ and regular $D$. Indeed, suppose $\tau^D$ were such an equilibrium. Here $(\mathcal{A}-r)g=2>0$, so by Proposition \ref{propeq} \eqref{a} it holds that $(D^c)^\circ = \varnothing$, in particular $D\neq \varnothing$. Now let $x\in D$. Since $\Ex_y[\tau^D]\le T$ for all $y\in \R$ the sets $D^c\cap (-\infty,x)$ and $D^c\cap (x,\infty)$ are non empty. Based on that we set $x_l:=\sup (D^c\cap (-\infty,x))$ and $x_r:=\inf (D^c\cap (x,\infty))$. As $D$ is open we infer $x_l<x<x_r$. Now $J_{\tau^D}(y)=x_r^2\frac{y-x_l}{x_r-x_l}+x_l^2 \frac{x_r-y}{x_r-x_l}$ for $y\in [x_l,x_r]$. Using that $J_{\tau^D} \ge g $ by Proposition \ref{propeq} \eqref{b} and $ J_{\tau^D}(x_r)=g(x_r)$ we infer
    \begin{align*}
        \partial_xJ_{\tau^D}(x_r-)
        &= x_l+x_r<2x_r 
        = \partial_x g(x_r+) 
        =\lim_{h\searrow 0} \frac{g(x_r+h)-g(x_r)}{h}\\
        &\le \lim_{h\searrow 0} \frac{J_{\tau^D}(x_r+h)- J_{\tau^D}(x_r)}{h}
        = \partial_x J_{\tau^D}(x_r+),
    \end{align*} which contradicts Proposition \ref{propeq} \eqref{d}.
\end{ex}

\begin{ex}\label{ex:BM_Betrag}
We now consider a Brownian motion $(W_t)_{t\in [0,\infty)}$ and the payoff function $g(x)=|x|$. Using the guess-and-verify approach from Remark \ref{scheme} we derive an equilibrium whose structure depends on the discount factor $r$ and if $r>0$ also on the size of $r\, T$. 
	
First we focus on $\boxed{r=0}$. In this case the process $g(W_t)=|W_t|$, $t\in [0,\infty)$, is a strict submartingale under every $\PR_x$.  
By Remark \ref{lambda_const} an equilibrium is given by $\tau^*=\tau^{D,\Psi}$, where $D=\R$ and $\psi\equiv\frac 1T$ defining $\Psi$ via \eqref{defpsi}.
%Hence, $\tau^*=\tau^{\Psi^*}$ which implies that $\tau^*$ is independent of the Brownian motion and $\tau^*\sim \text{Exp}\left(\frac 1T\right)$. In particular, $\Ex_x[\tau^*]=T$ for all $x\in \R$.  
The corresponding $J_{\tau^*}$  is given by 
    \begin{align*}%\label{ex_absolute_value}
	    J_{\tau^*}(x)=\Ex_x[|W_{\tau^*}|]
     %=\int_0^\infty \frac 1T \exp\left(-\frac t T\right)\Ex_x[|W_t|]\, dt
     =|x|+\sqrt{\frac T 2} \exp\left(-\sqrt{\frac{2}{T}}\,|x|\right)\!.
	\end{align*}
		
	For a positive discount factor $\boxed{r>0}$ we first consider the unconstrained optimal stopping problem 
\begin{align*}
	\sup_{\tau\in\mathcal{T}}\Ex_x[e^{-r\tau}|W_\tau|], 
\end{align*}
	where $\mathcal{T}$ denotes the set of all almost surely finite stopping times. With the usual methods (for more details we refer to  \cite{PeskirShiryaev2006}) one can show that the optimal stopping time is given by $\tau^{\tilde{D}}$, where $\tilde{D}=(-\tilde{x},\tilde{x}$) with $\tilde{x}=\tilde{x}(r)=\frac{\tilde{z}}{\sqrt{2r}}$ and $\tilde{z}$ is the unique positive solution of 
		\begin{align}\label{ex:eq_for_z}
		z=\frac{\cosh(z)}{\sinh(z)}. 
		\end{align}
	Note that $\tilde{z}\approx 1.19967864$.
	%
	%Indeed, let $\tau$ be an $\PR_x$-almost surely finite stopping time  and let 
	%\begin{align*}
	%	h_r(x)=\frac{|x|}{2\cosh(\sqrt{2r}\,x)}. 
	%\end{align*}
	%Then it holds that 
	%\begin{align*}
	%	\Ex_x\big[e^{-r\tau}|W_\tau|\big]&=\Ex_x\left[e^{-r\tau}\left(e^{\sqrt{2r}W_\tau}+e^{-\sqrt{2r}W_\tau}\right)h_r(W_\tau)\right]\\[0.2cm]
	%	&\leq \left(\max_{z\in\R} h_r(z)\right) \Ex_x\left[e^{-r\tau}\left(e^{\sqrt{2r}W_\tau}+e^{-\sqrt{2r}W_\tau}\right)\right]\\[0.2cm]
	%	&\leq \left(\max_{z\in\R} h_r(z)\right) 2\cosh(\sqrt{2r}\, x)\\[0.2cm]
    %	&= 2\cosh(\sqrt{2r}\, x)h_r(\tilde{x})\\[0.2cm]
    %	&=\frac{\tilde{x}\cosh(\sqrt{2r}\, x)}{\cosh(\sqrt{2r}\,\tilde{x})}.
	%	\end{align*}
	Moreover, we have 
    %one can show that % for all $x\in \tilde{D}$ we have
	%\begin{align*}
	%	\Ex_x\big[e^{-r\tau^{\tilde{D}}}\left|W_{\tau^{\tilde{D}}}\right|\big]
    %=	\tilde{x} \, \Ex_x\big[e^{-r\tau^{\tilde{D}}}\big] 
    %=\frac{\tilde{x}\cosh(\sqrt{2r}\, x)}{\cosh(\sqrt{2r}\,\tilde{x})} , 
	%\end{align*}
	%where we use Formula 3.0.1 in Chapter 2.1,  Part II of \cite{Borodin}.  Hence, $\tau^{\tilde{D}}$ is optimal for all $x\in \tilde{D}$. The smooth-fit principle and the associated free-boundary problem for the unconstrained optimal stopping problem (for more details we refer to  \cite{PeskirShiryaev2006}) imply that 
	\begin{align}
	\sup_{\tau\in\mathcal{T}} \Ex_x[e^{-r\tau}|W_\tau|]=\begin{cases}
		\begin{aligned}
		&\frac{\tilde{x}\cosh(\sqrt{2r}\, x)}{\cosh(\sqrt{2r}\,\tilde{x})},\qquad && |x|<\tilde{x}, \\[0.2cm]
		&|x|, &&|x|\geq\tilde{x}.
	\end{aligned}
	\end{cases}\label{ex_absolute_value_discount}
	\end{align}
%with optimal stopping time $\tau^{\tilde{D}}$. 	 
	%
	In particular, it holds that
	\begin{align*}
	\Ex_x\big[\tau^{\tilde{D}}\big]=\begin{cases}
		\begin{aligned}
		&\frac{\tilde{z}^2}{2r}-x^2, \qquad &&x \in \tilde{D},\\[0.2cm]
		& 0, && x\notin\tilde{D}. 
		\end{aligned}
		 \end{cases}
	\end{align*}
	Hence, 	$\Ex_x\big[\tau^{\tilde{D}}\big]\leq T$ for all $x\in \R$ if and only if $r\,T\geq \frac{\tilde{z}^2}{2}\approx 0.7196144195$. 
	
	Now we come back to the optimal stopping problem with expectation constraint. If $\boxed{r\,T\geq \frac{\tilde{z}^2}{2}}$ then the optimal stopping time $\tau^{\tilde{D}}$ of the unconstrained problem is admissible, i.e.\ $ \Ex_x\left[\tau^{\tilde{D}}\right]\leq T$ for all $x\in \R$. Let 
	\begin{align*}
		D^*=\tilde{D}=(-\tilde{x}, \tilde{x}), \qquad \psi\equiv 0
	\end{align*} 
and $\tau^*=\tau^{D^*}$. The corresponding $J_{\tau^*}$ is given by \eqref{ex_absolute_value_discount}. Then all conditions of Theorem~\ref{veri} are satisfied and $\tau^*$ is an equilibrium strategy. 
 
If $\boxed{r\,T< \frac{\tilde{z}^2}{2}}$ then the optimal strategy of the unconstrained optimal stopping problem does not fulfill the expectation constraint for all $x\in \R$ and, thus, is not admissible. 

To find an equilibrium  we follow Remark \ref{scheme} and assume that the strategy $\tau$ stops immediately under $\PR_x$ if $|x|$ is greater or equal to  $b$, continues if $|x|$ is sufficiently large but not greater than $b$ and stops with rate $\psi(x)=\frac 1T$ for small $|x|$. More precisely, we consider $\tau^{D^b, \Psi^a}$ with  $D^b=(-b,b)$ and $\psi^a=\frac{1}{T}\mathds{1}_{[-a,a)}$  defining $\Psi^a$ via \eqref{defpsi} for  $0<a<b$. Here $a$ and $b$ have to be chosen in such a way that the expectation constraint $\Ex_x[\tau^{D^b, \Psi^a}]\leq T $ holds true for all $x\in \R$.
Observe that \eqref{sys1} and \eqref{sys2} imply that 
\begin{align*}
    \frac 12 e_{D^b,\Psi^a}''&=-1\qquad\qquad  \text{on } (-b,-a)\cup(a,b),\\[0.1cm]
    e_{D^b,\Psi^a}(-b)&= e_{D^b,\Psi^a}(b)=0,\\[0.1cm]
      e_{D^b,\Psi^a}(-a)&= e_{D^b,\Psi^a}(a)=T, \\[0.1cm]
     e_{D^b,\Psi^a}'(-a)&= e_{D,\Psi^a}'(a)=0,
\end{align*}
which yields $a=b-\sqrt{T}$. In particular, $b>\sqrt{T}$. To simplify notation we set $\tau^b=\tau^{D^b, \Psi^{b-\sqrt{T}}}$ and $a=b-\sqrt{T}$.

Now we derive $J_{\tau^{b}}$. For $|x|\geq b$ we have $J_{\tau^b}(x)=|x|$. Moreover,  the symmetry of the Brownian motion $W$, of the payoff function $g$ and of the randomized Markovian time $\tau^b$  yield $J_{\tau^b}(x)=J_{\tau^b}(-x)$.  
By \cite[Theorem 5.9]{dynkin_markov} it holds that
\begin{itemize}
     \item $J_{\tau^{b}}\vert_{(-a,a)}\in\ccc^2((-a,a))$,
\item $J_{\tau^b}\vert_{(-b,-a)}\in\ccc^2((-b,-a))$,
\item  $J_{\tau^b}\vert_{(a,b)}\in\ccc^2((a,b))$.
\end{itemize} 
Lemma \ref{Kri16} implies that $J_{\tau^b} $ satisfies 
\begin{align}
\begin{split}
    \frac{1}{2} J_{\tau^{b}}''(x)-\left(r+\frac1T\right)J_{\tau^{b}}(x) &= -\frac{1}{T}|x|,\quad\ \ x\in(-a,a),\\[0.2cm]
    \frac{1}{2}J_{\tau^b}''(x)-r J_{\tau^b}(x)&= 0,\qquad \qquad  x\in (-b,-a) \cup (a,b).    
\end{split}\label{bsp:DGL}
\end{align}
The differential equations \eqref{bsp:DGL} together with  $J_{\tau^b}(x)=J_{\tau^b}(-x)$ result in
\begin{align*}
    J_{\tau^b}(x)=\begin{cases}
    \begin{aligned}
        &\xi\cosh\left(\gamma|x|\right)+\frac{|x|}{rT+1}+\frac{1}{\gamma(rT+1)}e^{-\gamma|x|}, &&|x|<a=b-\sqrt{T},\\[0.2cm]
        &\lambda e^{\sqrt{2r}|x|}+\nu e^{-\sqrt{2r}|x|}, \qquad &&|x|\in(a,b)=(b-\sqrt{T}, b)     
        \end{aligned}
        \end{cases}
    \end{align*}
    for some  $ \xi=\xi(T,r,b),$ $\lambda=\lambda(T,r,b)$, $\nu=\nu(T,r,b)$ and with $\gamma=\sqrt{2\left(r+\frac1T\right)}$.
    Since the function $J_{\tau^b}$ is continuous on $[-b,b]$ by Lemma \ref{Jstetigkeit} and even $J_{\tau^b}\vert_{(-b,b)}\in \ccc^1((-b,b))$ by direct computations we obtain boundary conditions to determine $\xi$, $\lambda$ and $\nu$ which depend on $b$, but are unique for fixed $r,T$ and $b$.
    %\begin{align*}
 %      \lambda e^{\sqrt{2r}b}+\nu e^{-\sqrt{2r}b}& =b,\\
 %      \lambda e^{\sqrt{2r}(b-\sqrt{T})}+\nu e^{-\sqrt{2r}(b-\sqrt{T})}&=\xi\cosh\left(\gamma (b-\sqrt{T})\right)+ \frac{b-\sqrt{T}}{rT+1}+\frac{1}{\gamma(rT+1)}e^{-\gamma(b-\sqrt{T})}\\
 %      \sqrt{2r}\left(\lambda e^{\sqrt{2r}(b-\sqrt{T})}-\nu e^{-\sqrt{2r}(b-\sqrt{T})}\right)&=\gamma\xi\sinh\left(\gamma (b-\sqrt{T})\right)+ \frac{1}{rT+1}-\frac{1}{(rT+1)}e^{-\gamma(b-\sqrt{T})}\\
 % \end{align*}
 
We now derive $b$ such that $\partial_xJ_{\tau^b}(b-)=\partial_xJ_{\tau^b}(b+)=1$ which guarantees that the second part of Condition (iii) in Theorem \ref{veri} is satisfied. 
A straight forward calculation shows that $\partial_xJ_{\tau^b}(b-)=1$ holds if and only if  $b$ is a zero of the function $\ell\colon [\sqrt{T}, \infty)\to \R$ with  
\begin{align*}
    \ell(b)&=
\gamma\sinh(\gamma(b-\sqrt{T}))\left[b\cosh(\sqrt{2rT})-\frac{1}{\sqrt{2r}}\sinh(\sqrt{2rT})-\frac{b-\sqrt{T}}{rT+1}-\frac{e^{-\gamma(b-\sqrt{T})}}{\gamma(rT+1)}\right]\\
&\quad  +\cosh(\gamma(b-\sqrt{T}))\left[ \sqrt{2r} b \sinh(\sqrt{2rT})-\cosh(\sqrt{2rT})+\frac{1-e^{-\gamma(b-\sqrt{T})}}{rT+1}\right].
\end{align*}
One can show that $\ell$ is strictly increasing with $\lim_{b\to\infty} \ell(b)=\infty$ and 
\begin{align*}
    \ell(\sqrt{T})=\sqrt{2rT}\sinh(\sqrt{2rT})-\cosh(\sqrt{2rT})<0, 
    \end{align*}
where we use for the inequality that  $\sqrt{2rT}< \tilde{z}$ and $y\sinh(y)-\cosh(y)<0$ for all $y\in [0,\tilde{z})$.
Hence, there exists a unique  $b^*=b^*(T,r)\in (\sqrt{T}, \infty)$ with $\ell(b^*)=0$ which implies that $\partial_xJ_{\tau^{b^*}}(b^*-)=1$.

In addition, we have 
\begin{align*}
	J_{\tau^{b^*}}(x)=
	\begin{cases}
		\begin{aligned}
			& \xi(T,r, b^*)\, \cosh\left(\gamma |x|\right) +\frac{|x|}{rT+1}+\frac{1}{\gamma(rT+1)}e^{-\gamma|x|}, \\
   &\hspace*{9cm}|x|\leq b^*-\sqrt{T},\\[0.5cm]
			&\frac{1}{2}\left(b^*+\frac{1}{\sqrt{2r}}\right)\,e^{-\sqrt{2r}\,(b^*-|x|)}+\frac{1}{2}\left(b^*-\frac{1}{\sqrt{2r}}\right)\,e^{\sqrt{2r}\,(b^*-|x|)},\\
    &\hspace*{9cm} |x|\in (b^*-\sqrt{T}, b^*),\\[0.5cm]
			&|x|, \hspace{8.35cm} |x|\geq b^*, 
		\end{aligned}
	\end{cases}
\end{align*}
where
\begin{align*}
\xi(T,r, b^*)=\frac{1}{\cosh(\gamma(b^*-\sqrt{T}))}&\left(
b^*\cosh(\sqrt{2rT})-\frac{1}{\sqrt{2r}}\sinh(\sqrt{2rT})\right.\\
&\qquad\left.-\frac{1}{rT+1}(b^*-\sqrt{T}+\frac{1}{\gamma}e^{-\gamma(b^*-\sqrt{T})}\right).
\end{align*}
Now one can check that $\tau^{b^*}$ is indeed an equilibrium using  Theorem \ref{veri}.

See Figure \ref{fig:betrag_r=0} for the function $J_{\tau^*}$  for different $T>0$ with $r=0$ and Figure \ref{fig:J_betrag_small_T_pos_r} for  $J_{\tau^{b^*}}$ for  $r=0.01$ and different $T>0$.

%%%%% Figures
{\begin{figure}[h]
\begin{center}
\hspace*{-0.57cm}\includegraphics[width=0.54\textwidth]{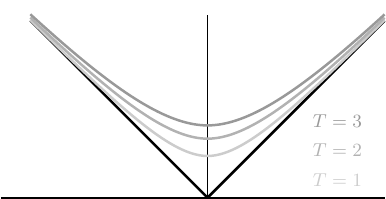}
\end{center}
\caption{The function $g(x)=|x|$ in black and the function $J_{\tau^*}$ for different $T>0$ and $r=0$. }
\label{fig:betrag_r=0}
%huhu huhu zurück
 %\begin{center}
%\hspace*{0.6cm}\includegraphics[scale=1.]{Bilder/Betrag_pos_r_unconstrained_sw.pdf}
%\end{center}
%  \caption{The value function of the unconstrained optimal stopping problem for the payoff function $g(x)=|x|$ with discount factor $r=0.01$  which coincides with $J_{\tau^*}$ for $T\geq \frac{ \tilde{z}^2}{2r}\approx71.9615$.}
%   \label{fig:J_betrag_unconstrained_pos_r}
   {\begin{center}
    \includegraphics[scale=1.1]{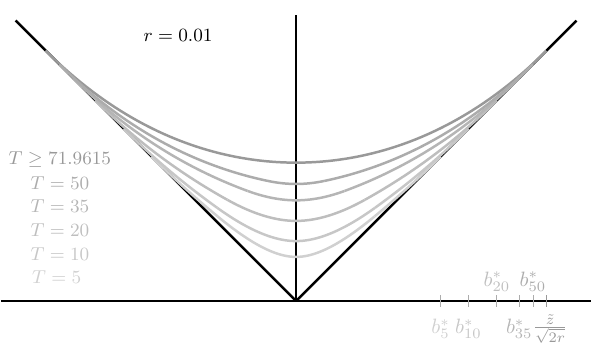}
    \end{center}}
    \caption{The function $J_{\tau^{*}}$  for the payoff function $g(x)=|x|$ (black), discount factor $r=0.01$ and $T\in\{5,10,20,35,50\}$ as well as $T\geq \frac{ \tilde{z}^2}{2r}\approx71.9615$. Here we use $b_T^*$ for $b^*=b^*(T,r)$.}    \label{fig:J_betrag_small_T_pos_r}
\end{figure}}
\end{ex}

\begin{remark}
    The most common formulation of the constrained optimal stopping problem for the Brownian motion $W$ and the payoff function $g(x)=|x|$ is
    \begin{align}
        \text{maximize } \Ex_x[|W_\tau|] \qquad \text{subject to $\tau\in \widetilde{\mathcal{S}}(x,T)$}, \label{eq_comment_precommitment}
    \end{align}
     for given $x\in \R$, where $\widetilde{\mathcal{S}}(x,T)$ denotes the set of all stopping times with respect to the augmented natural filtration of the Brownian motion satisfying  $\Ex_x[\tau]\leq T$. A major difference to our equilibrium approach is that the constraint only applies for the fixed starting value $x$ but not uniformly on the entire state space.
     
     An optimal stopping time in \eqref{eq_comment_precommitment} for fixed $x$ and $T$ is given by $\tau^{D(x,T)}$ with $D(x,T)=(-\sqrt{T+x^2}, \sqrt{T+x^2})$ and the value of \eqref{eq_comment_precommitment} is $\sqrt{T+x^2}$, see \cite[Example 4.5]{AKK}. In particular, for any $y$ with $|y|< |x|$ the stopping time $\tau^{D(x,T)}$ does not satisfy the expectation constraint $\Ex_y\big[\tau^{D(x,T)}\big]\leq T$ that would be required in the equilibrium problem. 
    
    Observe that in \eqref{eq_comment_precommitment} we do not allow for an external source of randomization. 
    
    The precommitment approach presented in \cite{AKK, miller2017} does not cover the discounted optimal stopping problem  
      \begin{align}\label{eq_comment_precommitment_discount}
        \text{maximize } \Ex_x[e^{-r\tau}|W_\tau|] \qquad \text{subject to $\tau\in \widetilde{\mathcal{S}}(x,T)$}
    \end{align}
    for $r>0$, which we treat in our equilibrium framework. Although \cite{BayraktarYao2023} comprises  Problem \eqref{eq_comment_precommitment_discount}, to the best of our knowledge there are no solution methods at hand to derive an optimal stopping time or  the value of \eqref{eq_comment_precommitment_discount}.
      
\end{remark}

Finally, we present an example where in an equilibrium $\tau^{D,\Psi}$ the set $D$ is not an open interval but an open set.
\begin{ex}\label{ex:D_kein_Intervall}
    Consider the payoff function 
    \begin{align*}
        g(x)=\begin{cases}
            |x-2|, & \text{if $x\in(1,3)$,}\\
            |x+2|, & \text{if $x\in (-3,-1)$,}\\
            1, &\text{else,}
        \end{cases}
    \end{align*}
    see Figure \ref{fig:D_kein_Intervall}. 
    Example \ref{ex:BM_Betrag} implies that for a  positive discount factor $r$ such that  $\widetilde{x}=\frac{\tilde{z}}{\sqrt{2r}}\leq 1$  and  $T\geq \frac{\tilde{z}^2}{2r}$,  where $\widetilde{z}$ is the unique positive solution of \eqref{ex:eq_for_z}, an equilibrium is given by  $\tau^{D,\Psi}$ with 
    $D=(-2-\widetilde{x},-2+\widetilde{x})\cup (2-\widetilde{x},2+\widetilde{x})$ and $\psi\equiv 0.$ 
       \begin{figure}[h]
    \begin{center}
        \includegraphics{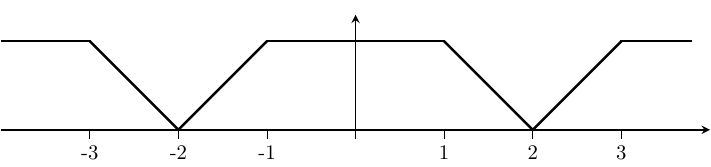}
        \caption{\textnormal{The payoff function in Example \ref{ex:D_kein_Intervall}.}}\label{fig:D_kein_Intervall}
        \end{center}
    \end{figure}
    \end{ex}

\appendix
\section{Appendix}
\begin{lemma}\label{rightcont} 
\begin{enumerate}[(i)]
    \item \label{rightcont_1} For every randomized Markovian time $\tau^{D,\Psi}\in \mathcal{S}(T)$ the expected time function $e_{D,\Psi}$ is continuous. \label{p1}
    \item For every $x\in (\alpha, \beta)$ there is some neighborhood $[x-\eta,x+\eta]$ with $\eta>0$ such that
\begin{align*}
    \Ex_y[\tau_h] \stackrel{h \searrow0 }{\to} 0 
\end{align*} uniformly in $y\in [x-\eta, x+\eta]$. \label{p2}
\end{enumerate}

\end{lemma}
\begin{proof}
Since $\edp=0$ outside $D$ it is sufficient to consider $x\in \overline{D}$. Let $a<b\in \overline{D}$ such that $x\in[a,b]$. 
Let $X^{\tau^{\Psi}}=(X^{\tau^{\Psi}}_t)_{t\in [0,\infty)}$ be given by 
\begin{align*}
    X^{\tau^{\Psi}}_t:=\begin{cases}
        X_t, &t<\tau^{\Psi},\\
        \partial, & t\ge \tau^{\Psi}
    \end{cases}
\end{align*} with a coffin state $\partial$. Additionally we set $\mathfrak m_y:=\inf\{t\ge 0 : X^{\tau^{\Psi}}_t=y\}$, $y\in (\alpha,\beta)$ and $\mathfrak m_\infty:=\inf\{t\ge 0 : X^{\tau^{\Psi}}_t=\partial\}$. By \cite[p.\ 119,  p.\ 121]{ito1974diffusion} the functions $x\mapsto \PR_x(\mathfrak m_a <\mathfrak m_b) = \PR_x(X_{\tau^{(a,b)} \wedge\tau^\Psi}=a )$, $x\mapsto \PR_x(\mathfrak m_a >\mathfrak m_b) = \PR_x(X_{\tau^{(a,b)} \wedge\tau^\Psi}=b )$ and $x\mapsto \Ex_x[\mathfrak m_a \wedge \mathfrak m_b \wedge \mathfrak m_\infty] = \Ex_x[\tau^{(a,b)} \wedge\tau^\Psi]$ are continuous.
The Markov property yields
\begin{align*}
    \edp(x) = \Ex_x[\tau^{(a,b)} \wedge\tau^\Psi] + \edp(a) \PR_x(X_{\tau^{(a,b)} \wedge\tau^\Psi}=a )+\edp(b) \PR_x(X_{\tau^{(a,b)} \wedge\tau^\Psi}=b),
\end{align*} which proves the first claim.

For the second claim let $f:(\alpha,\beta)\to\R$ solve $\mathcal{A}f=1$. By Dynkin's formula
\begin{align*}
    \Ex_y[\tau_h]=\Ex_y\left[ \int_0^{\tau_h} \mathcal{A}f(X_s)ds \right] = \Ex_y[f(X_{\tau_h})]-f(y)\le \sup_{z_1,z_2\in[y-h,y+h]} \vert f(z_1)-f(y_2)\vert.
\end{align*} Uniform continuity of $f$ on any $[x-\eta,x+\eta]\subset (\alpha,\beta)$ provides the claim.
\end{proof}

\begin{lemma} \label{Kri10}
Let $\tau^{D, \Psi}, \tau^{\tilde{D},\tilde{\Psi}} \in \mathcal{S}$ be randomized Markovian times and $\psi\in \mathfrak{V}_D$, $\tilde{\psi}\in \mathfrak{V}_{\tilde{D}}$ the functions defining $\Psi$ and $\tilde{\Psi}$ via \eqref{defpsi}, respectively. Suppose that $g$ has real left and right limits for all $x\in D$. For all $x\in D \cap\tilde{D}$ such that $J_{\tau^{D,\Psi}}$ is continuous in $x$ we have
\begin{align}
    \lim_{h \searrow 0}\frac{ J_{ \theta_{\tau_h}\circ\tau^{D, \Psi} +\tau_h}(x)  -J_{\tau^{D, \Psi}\diamond \tau^{\tilde{D},\tilde{\Psi}}(h)}(x)}{\Ex_x[\tau_h]}=&\, \frac{1}{2}\tilde{\psi}(x-)(J_{\tau^{D,\Psi}}(x)-g(x-) )\notag\\
    &+\frac{1}{2} \tilde{\psi}(x)(J_{\tau^{D,\Psi}}(x)- g(x+) ). \notag
\end{align} 
\end{lemma}

\begin{beweis}
Note that by the Markov property
\begin{align}
    &J_{\theta_{\tau_h} \circ\tau^{D, \Psi}  +\tau_h} (x) =\Ex_{x}[e^{-r\tau_h} J_{\tau^{D,\Psi}}(X_{\tau_h}) ] \label{aux1}
\end{align}
and 
\begin{align}
    &J_{\tau^{D, \Psi}\diamond \tau^{\tilde{D},\tilde{\Psi}}(h)} (x) =\Ex_{x}\big[ \ind{\tau^{\tilde{D},\tilde{\Psi}} \le \tau_h } e^{-r \tau^{\tilde{D},\tilde{\Psi}}} g(X_{\tau^{\tilde{D},\tilde{\Psi}}}) + \ind{\tau^{\tilde{D},\tilde{\Psi}} > \tau_h}e^{-r \tau_h} J_{\tau^{D,\Psi}}(X_{\tau_h})\big] \label{aux2}.
\end{align}

Let $h>0$ such that $[x-h,x+h] \subset \tilde{D}$. This leads to $\tau^{\tilde{D},\tilde{\Psi}} = \tau^{\tilde{\Psi}}$ on $\{ \tau^{\tilde{D},\tilde{\Psi}} \le \tau_h\}$. Now using \eqref{aux1}, \eqref{aux2} and Remark \ref{rem1} \eqref{conddistr} we obtain
\begin{align}
    &J_{\theta_{\tau_h} \circ \tau^{D, \Psi} +\tau_h} (x)-J_{\tau^{D, \Psi}\diamond \tau^{\tilde{D},\tilde{\Psi}}(h)} (x)  \notag\\[0.1cm]
    & =  \Ex_{x}\big[ \ind{\tau^{\tilde{D},\tilde{\Psi}} \le \tau_h } \big(e^{-r \tau_h} J_{\tau^{D,\Psi}}(X_{\tau_h}) - e^{-r \tau^{\tilde{D},\tilde{\Psi}}} g(X_{\tau^{\tilde{D},\tilde{\Psi}}})\big)\big]\notag\\[0.1cm]
    %=&\Ex_{x}[\Ex_x[ \ind{\tau^{\tilde{D},\tilde{\Psi}} \le \tau_h } (e^{-r \tau_h} J_{\tau^{D,\Psi}}(X_{\tau_h}) - e^{-r \tau^{\tilde{D},\tilde{\Psi}}} g(X_{\tau^{\tilde{D},\tilde{\Psi}}})  ) \vert \f^X_{\infty}]]\notag\\
    &= \Ex_{x}\left[\Ex_x\left[ \ind{\tau^{\tilde{\Psi}} \le \tau_h } \big(e^{-r \tau_h} J_{\tau^{D,\Psi}}(X_{\tau_h}) - e^{-r \tau^{\tilde{\Psi}}} g(X_{\tau^{\tilde{\Psi}}}) \big) \vert \f^X_{\infty}\right]\right]\notag\\[0.1cm]
    %=& \Ex_{x}\bigg[ \int_0^{\tau_h}  (e^{-r \tau_h} J_{\tau^{D,\Psi}}(X_{\tau_h}) - e^{-rt} g(X_{t})  )  d \PR_x(\tau^{\tilde{\Psi}}\le t \vert \f_\infty^X) \bigg]\notag\\
    &= \Ex_{x}\bigg[ \int_0^{\tau_h}  \big(e^{-r \tau_h} J_{\tau^{D,\Psi}}(X_{\tau_h}) - e^{-rt} g(X_{t}) \big ) \tilde{\psi}(X_t)e^{-\int_0^t \tilde{\psi}(X_s)ds} dt \bigg]\notag\\[0.1cm]
    &=   \Ex_{x}\left[ \int_0^{\tau_h}   J_{\tau^{D,\Psi}}(X_{\tau_h}) \tilde{\psi}(X_t) \left(e^{-\int_0^t \tilde{\psi}(X_s)ds-r \tau_h}-1\right) dt \right]  +  \Ex_x \left[\int_0^{\tau_h}J_{\tau^{D,\Psi}}(X_{\tau_h}) \tilde{\psi}(X_t) dt\right] \notag\\[0.1cm]
    &\quad -   \Ex_{x}\left[ \int_0^{\tau_h}   g(X_{t})  \tilde{\psi}(X_t)\left( e^{-\int_0^t \tilde{\psi}(X_s)ds-rt}-1\right) dt \right]  - \Ex_{x}\left[ \int_0^{\tau_h}   g(X_{t})  \tilde{\psi}(X_t) dt \right] . \label{16}
\end{align}
We now divide \eqref{16} by $\Ex_x[\tau_h]$ and evaluate the limits of all the summands on the new right-hand side in order to prove the claim. The first and third summand can be estimated as follows using $e^{-z}-1 \ge -z$ for all $z \ge 0$.
\begin{align*}
    & \Bigg\vert  \frac{ \Ex_{x}\left[ \int_0^{\tau_h}   J_{\tau^{D,\Psi}}(X_{\tau_h}) \tilde{\psi}(X_t) \left(e^{-\int_0^t \tilde{\psi}(X_s)ds-r \tau_h}-1\right) dt \right] }{\Ex_x[\tau_h]} \\[0.1cm]
    & -\frac{  \Ex_{x}\left[ \int_0^{\tau_h}   g(X_{t})  \tilde{\psi}(X_t)\left( e^{-\int_0^t \tilde{\psi}(X_s)ds-rt}-1\right) dt \right]}{\Ex_x[\tau_h]}\Bigg\vert \\[0.1cm]
    %\le& \frac{  \Ex_{x}\left[ \int_0^{\tau_h}   g(X_{t})  \tilde{\psi}(X_t)(\int_0^t \tilde{\psi}(X_s)ds+rt)dt \right]  }{\Ex_x[\tau_h]} \\
    %+&  \frac{ \Ex_{x}\left[ \int_0^{\tau_h}   J_{\tau^{D,\Psi}}(X_{\tau_h}) \tilde{\psi}(X_t) (\int_0^t \tilde{\psi}(X_s)ds+r \tau_h) dt \right] }{\Ex_x[\tau_h]}\\
    \le &  \,\frac{  \Ex_{x}\left[ \int_0^{\tau_h}  ( g(X_{t})+ J_{\tau^{D,\Psi}}(X_{\tau_h}))  \tilde{\psi}(X_t)(\int_0^{\tau_h} \tilde{\psi}(X_s)ds+r\tau_h)dt \right]  }{\Ex_x[\tau_h]} \\[0.1cm]
    \le& \,\sup_{y_1,y_2,y_3,y_4\in [x-h,x+h]} (g(y_1)+ J_{\tau^{D,\Psi}}(y_2))\tilde{\psi}(y_3) ( \tilde{\psi}(y_4) + r ) \frac{\Ex_x[\tau_h^2]}{\Ex_x[\tau_h]} \stackrel{h \searrow 0}{\to} 0.
\end{align*}
For the proof of $\frac{\Ex_x[\tau_h^2]}{\Ex_x[\tau_h]} \stackrel{h \searrow 0}{\to} 0$ we refer to \cite[Lemma 25]{BodnariuChristensenLindensjoe2022}. With Lemma \ref{appL1} from the appendix we can determine the limit of the fourth summand of the right-hand side of \eqref{16} divided by $\Ex_x[\tau_h]$:
\begin{align*}
    \frac{\Ex_{x}\left[ \int_0^{\tau_h}   g(X_{t})  \tilde{\psi}(X_t) dt \right]}{\Ex_x[\tau_h]} \stackrel{h \searrow 0}{\to}\frac{1}{2}(\tilde{\psi}(x-)g(x-)+\tilde{\psi}(x+)g(x+)).
\end{align*}
For the second summand of \eqref{16} over $\Ex_x[\tau_h]$ we have
\begin{align*}
    &\left \vert \frac{\Ex_x[\int_0^{\tau_h}J_{\tau^{D,\Psi}}(X_{\tau_h})\tilde{\psi}(X_t)dt] }{\Ex_x[\tau_h]}-\frac{1}{2} J_{\tau^{D,\Psi}}(x) (\tilde{\psi}(x-) + \tilde{\psi}(x+)) \right \vert\\[0.1cm]
    \le& \left \vert \frac{\Ex_x[\int_0^{\tau_h}J_{\tau^{D,\Psi}}(X_{\tau_h})\tilde{\psi}(X_t)dt] }{\Ex_x[\tau_h]}-\frac{\Ex_x[\int_0^{\tau_h}J_{\tau^{D,\Psi}}(x)\tilde{\psi}(X_t)dt] }{\Ex_x[\tau_h]} \right \vert \\[0.1cm]
    &+ \left \vert J_{\tau^{D,\Psi}}(x)\frac{\Ex_x[\int_0^{\tau_h}\tilde{\psi}(X_t)dt] }{\Ex_x[\tau_h]}-\frac{1}{2} J_{\tau^{D,\Psi}}(x) (\tilde{\psi}(x-) + \tilde{\psi}(x+)) \right \vert\\[0.1cm]
    \le &  \sup_{y\in[x-h,x+h]}\vert J_{\tau^{D,\Psi}}(y)-J_{\tau^{D,\Psi}}(x)\vert \frac{\Ex_x[\int_0^{\tau_h}\tilde{\psi}(X_t)dt] }{\Ex_x[\tau_h]} \\[0.1cm]
    &+ \vert  J_{\tau^{D,\Psi}}(x)\vert \left \vert \frac{\Ex_x[\int_0^{\tau_h}\tilde{\psi}(X_t)dt] }{\Ex_x[\tau_h]}-\frac{1}{2}  (\tilde{\psi}(x-) + \tilde{\psi}(x+))\right \vert.
\end{align*}
By Lemma \ref{appL1} and the continuity of $J_{\tau^{D,\Psi}}$ in $x$ the right-hand side of this inequality goes to 0 for $h\searrow0$. Invoking the left-hand side of \eqref{16} as well as the right continuity of $\psi$ this finishes the proof.
\end{beweis}

\begin{lemma}\label{Jstetigkeit}
Let $\tau^{D, \Psi} \in \mathcal{S}(T)$ be a randomized Markovian time and $\psi \in \mathfrak{V}_D$ the function defining $\Psi$ via \eqref{defpsi}. Let $D' $ be some connected component of $D$. If $g\vert_{\overline{D'}}$ is bounded, then $J_{\tau^{D,\Psi}}\vert_{\overline{D'}}$ is continuous.
\end{lemma}
\begin{beweis}
Let $E\sim \Exp(r)$ be a random variable on $(\Omega,\f,\PR)$ that is independent of $X$ and $N$ and let $\hat{X}=(\hat{X}_t)_{t \in [0,\infty)}$ denote the process $X$ killed with $E$, i.e.
\begin{align*}
    \hat{X}_t(\omega):= \begin{cases}
        X_t(\omega) &\text{if }  t<E(\omega),\\
        \infty &\text{if } t \ge E(\omega).
    \end{cases}
\end{align*} We extend the function $g$ to $\infty$ by setting $g(\infty):= 0$. Now for $\inf D' <a<x<b<\sup D'$ the Markov property gives
\begin{align*}
    J_{\tau^{D,\Psi}}(x) &= \Ex_x[ g(\hat{X}_{\tau^{D,\Psi}})] \\
    &= \PR_x(\hat{X}_{\tau^{(a,b)}\wedge\tau^\Psi}=a) \Ex_a[ g(\hat{X}_{\tau^{D,\Psi}})] + \Ex_x[\ind{\hat{X}_{\tau^{(a,b)}\wedge\tau^\Psi}\neq a} g(\hat{X}_{\tau^{D,\Psi}})]\\
    &\stackrel{x\searrow a}{\to} \Ex_a[ g(\hat{X}_{\tau^{D,\Psi}})] = J_{\tau^{D,\Psi}}(a)
\end{align*} since $\PR_x(\hat{X}_{\tau^{(a,b)}\wedge\tau^\Psi}=a) \to 1$ for $x\to a$ by \cite[p.\ 119]{ito1974diffusion} as in the proof of Lemma \ref{rightcont} and thus also 
\begin{align*}
    \big \vert \Ex_x[\ind{\hat{X}_{\tau^{(a,b)}\wedge\tau^\Psi}\neq a} g(\hat{X}_{\tau^{D,\Psi}})] \big \vert 
    \le (1-\PR_x(\hat{X}_{\tau^{(a,b)}\wedge\tau^\Psi}=a)) \sup_{y\in \overline{D'}} g(y)\to 0.
    \end{align*}
    The case $x \nearrow b$ is analogous.
\end{beweis} 

\begin{lemma}\label{appL1}
%We consider the process $X$ and the stopping times introduced in Section \ref{Ch2}.\\
For every $x\in (\alpha,\beta)$ and each function $f\colon (\alpha,\beta) \to \R$ such that $f(x+)$ and $f(x-)$ exist (in $\R$), we have
\begin{align}
    \lim_{h \searrow 0} \frac{ \Ex_x[ \int_0^{\tau_h}f(X_s) ds ]}{\Ex_x[\tau_h]} = \frac{1}{2}(f(x-) + f(x+)).\label{nervmichnichtbitte}
\end{align}
\end{lemma}

\begin{proof}
    Without loss of generality we may assume that $\mu,\sigma$ are bounded and $\sigma$ is also bounded away from $0$ as well as $\alpha=-\infty$ and $\beta=\infty$, since \eqref{nervmichnichtbitte} only depends on $X$ in a neighborhood of $x$.
    First we show that in order to prove \eqref{nervmichnichtbitte} it suffices to verify
\begin{align}
    \lim_{h \searrow 0}& \frac{ \Ex_x[ \int_0^{\tau_h}\ind{X_s>x} ds ]}{\Ex_x[\tau_h]} =\lim_{h \searrow 0} \frac{ \Ex_x[ \int_0^{\tau_h}\ind{X_s<x} ds ]}{\Ex_x[\tau_h]} = \frac{1}{2},\label{reichtaus1}
    \intertext{and}
     &\Ex_x\left[ \int_0^{\tau_h}\ind{X_s=x} ds \right] = 0 \,\ttt{ for all }h>0. \label{reichtaus2}
\end{align}
For this purpose observe
\begin{align*}
    \Ex_x\left[\int_0^{\tau_h} f(X_s)ds\right] =& \Ex_x\left[\int_0^{\tau_h}\ind{X_s>x} f(X_s)ds\right]+\Ex_x\left[\int_0^{\tau_h} \ind{X_s<x} f(X_s)ds\right]\\
    &+\Ex_x\left[\int_0^{\tau_h} \ind{X_s=x}f(X_s)ds\right].
\end{align*}
For $\eps >0$ we choose $h>0$ such that $\vert f(y)-f(x-)\vert< \eps$ for all $y\in [x-h,x)$ and $\vert f(y)-f(x+)\vert <\eps $ for all $y \in (x,x+h]$. With this we get 
\begin{align*}
   &\left\vert \frac{\Ex_x\left[\int_0^{\tau_h}\ind{X_s>x} f(X_s)ds\right]}{\Ex_x[\tau_h]} - \frac{1}{2}f(x+)\right\vert\\[0.1cm]
   \le& \left\vert \frac{\Ex_x\left[\int_0^{\tau_h}\ind{X_s>x} f(X_s)ds\right]}{\Ex_x[\tau_h]} - \frac{\Ex_x\left[\int_0^{\tau_h}\ind{X_s>x} f(x+)ds\right]}{\Ex_x[\tau_h]} \right \vert \\[0.1cm]
   &+ \left \vert \frac{\Ex_x\left[\int_0^{\tau_h}\ind{X_s>x} f(x+)ds\right]}{\Ex_x[\tau_h]} - \frac{1}{2}f(x+)  \right \vert\\[0.1cm]
   =& \,\frac{\Ex_x\left[\int_0^{\tau_h}\ind{X_s>x} \vert f(X_s)-f(x+)\vert ds\right]}{\Ex_x[\tau_h]} + \vert f(x+) \vert \cdot \left \vert \frac{\Ex_x\left[\int_0^{\tau_h}\ind{X_s>x} ds\right]}{\Ex_x[\tau_h]} - \frac{1}{2}  \right \vert\\[0.1cm]
   \le&\, \eps + \vert f(x+) \vert \cdot \left \vert \frac{\Ex_x\left[\int_0^{\tau_h}\ind{X_s>x} ds\right]}{\Ex_x[\tau_h]} - \frac{1}{2}  \right \vert %\stackrel{\eps\searrow0}{\to}\vert f(x+) \vert \cdot \left \vert \frac{\Ex_x\left[\int_0^{\tau_h}\ind{X_s>x} ds\right]}{\Ex_x[\tau_h]} - \frac{1}{2}  \right \vert.
\end{align*}
An analogous estimate with $\ind{X_s>x}$ and $f(x+)$ replaced by $\ind{X_s<x}$ and $f(x-)$, respectively, together with $\Ex_x\left[\int_0^{\tau_h} \ind{X_s=x}f(X_s)ds\right]=f(x) \Ex_x\left[\int_0^{\tau_h} \ind{X_s=x}ds\right]$ shows that \eqref{reichtaus1} and \eqref{reichtaus2} are in fact sufficient for \eqref{nervmichnichtbitte}. 

The transition densities $p_t(x,y):=\PR_x(X_t\in dy)$ can be estimated as follows:
\begin{align}
    p_t(x,y) \le \frac{K_1}{\sqrt{t}}\exp\left(-K_2\frac{(y-x)^2}{t}\right),\label{dichteabsch}
\end{align} 
cf.\ \cite[Equation (1.2)]{Shuenn}.
With this we immediately obtain \eqref{reichtaus2} via
\begin{align*}
    0 &\le \Ex_x\left[ \int_0^{\tau_h}\ind{X_s=x} ds \right] \le \Ex_x\left[ \int_0^{\infty}\ind{X_s=x} ds \right] = \int_0^{\infty}\Ex_x\left[ \ind{X_s=x}  \right]ds \\[0.1cm]
    &\le \int_0^{\infty} \int_{\R} \ind{y=x}\frac{K_1}{\sqrt{s}}\exp\left(-K_2\frac{(y-x)^2}{s}\right) dy ds=0.
\end{align*}
Concerning \eqref{reichtaus1} we only show $\lim_{h \searrow 0} \frac{ \Ex_x[ \int_0^{\tau_h}\ind{X_s>x} ds ]}{\Ex_x[\tau_h]} =\frac{1}{2}$ as for the other term the argument is exactly the same. By \cite[Theorem 5.5]{dynkin_markov} and Fubini's theorem
\begin{align*}
    \lim_{h \searrow 0} \frac{ \Ex_x[ \int_0^{\tau_h}\ind{X_s>x} ds ]}{\Ex_x[\tau_h]} = \lim_{t\searrow0}\frac{\Ex_x[\int_0^t\ind{X_s>x}ds]}{t}= \lim_{t\searrow0}\frac{\int_0^t\PR_x(X_s>x)ds}{t}.
\end{align*} By \cite[Lemma 5.5]{bayraktar2022equilibria} $\lim_{s \searrow 0} \PR_x(X_s>x)=\frac{1}{2}$ and thus
\begin{align*}
    \lim_{t\searrow0}\frac{\int_0^t\PR_x(X_s>x)ds}{t}=\frac{1}{2}
\end{align*} which proves the claim.
\end{proof}

\bibliographystyle{abbrv}
\bibliography{literature}

\end{document}